\documentclass[11pt]{article}
\usepackage[letterpaper, margin=1in]{geometry}
\usepackage[utf8]{inputenc}

\usepackage{amsmath, amssymb, amsthm, authblk, bbm, mathtools, subcaption}
\usepackage[sort&compress,numbers]{natbib}
\usepackage[linktocpage=true]{hyperref}
\hypersetup{colorlinks=true, linkcolor=blue}

\numberwithin{equation}{section}
\newtheorem{theorem}{Theorem}[section]
\newtheorem{corollary}[theorem]{Corollary}
\newtheorem{definition}[theorem]{Definition}
\newtheorem{proposition}[theorem]{Proposition}
\newtheorem{lemma}[theorem]{Lemma}
\newtheorem{remark}[theorem]{Remark}

\newcommand{\keywords}[1]{\textbf{\textit{Keywords ---}} #1}
\newcommand*\diff{\mathop{}\!\mathrm{d}}

\renewcommand{\P}{\mathbb{P}}
\newcommand{\E}{\mathbb{E}}

\newcommand{\N}{\mathbb{N}}

\newcommand{\bq}{\boldsymbol{q}}
\newcommand{\bX}{\boldsymbol{X}}
\newcommand{\bY}{\boldsymbol{Y}}

\newcommand{\cF}{\mathcal{F}}
\newcommand{\cN}{\mathcal{N}}

\begin{document}

\title{Mean-field Analysis for Load Balancing on Spatial Graphs}

\author[1]{Daan Rutten\thanks{Email: \href{mailto:drutten@gatech.edu}{drutten@gatech.edu}}\textsuperscript{,}}
\author[1]{Debankur Mukherjee}

\affil[1]{Georgia Institute of Technology}

\maketitle

\begin{abstract}
The analysis of large-scale, parallel-server load balancing systems has relied heavily on mean-field analysis. A pivotal assumption for this framework is that the servers are exchangeable. However, modern data-centers process a wide variety of task types and the data required to process tasks of a certain type is stored locally at servers. This gives rise to \emph{data locality constraints}, where tasks of a particular type can only be routed to a small subset of servers. An emerging line of research, therefore, considers load balancing algorithms on bipartite graphs where vertices in the two partitions represent the task types and servers, respectively, and an edge represents the server's ability to process the corresponding task type. Due to the lack of exchangeability in this model, the mean-field techniques fundamentally break down. Recent progress has been made by considering graphs with strong edge-expansion properties, i.e., where \emph{any} two large subsets of vertices are well-connected. However, data locality often leads to spatial constraints, where edges are local. As a result, these bipartite graphs do not have strong expansion properties.

In this paper, we consider the power-of-$d$ choices algorithm and develop a novel coupling-based approach to establish mean-field approximation for a large class of graphs that includes spatial graphs.
As a corollary, we also show that, as the system size becomes large, the steady-state occupancy process for \emph{arbitrary} regular bipartite graphs with diverging degrees, is indistinguishable from a fully flexible system on a complete bipartite graph. The method extends the scope of mean-field analysis far beyond the classical full-flexibility setup. En route, we prove that, starting from suitable states, the occupancy process becomes close to its steady state in a time that is independent of $N$. Such a large-scale mixing-time result might be of independent interest.
Numerical experiments are conducted, which positively support the theoretical results.
\end{abstract}

\section{Introduction}
\noindent
\textbf{Background and motivation.}
The study of load balancing algorithms for large-scale systems started with the seminal works of Mitzenmacher \cite{Mitzenmacher96} and Vvedenskaya et al.~\cite{VDK96}.
Since then, there has been a significant development in our understanding of the performance of various load balancing policies and their tradeoffs between quantities like user-perceived delay, communication overhead, implementation complexity, and energy consumption; see for example~\cite{lu2011join, weng2020achieving, comte2019dynamic, gupta2019load, anton2020improving, anton2021stability, van2021load, choudhury2021job, panigrahy2022analysis, jonckheere2022generalized, BBLM21} for a few recent, representative works from various research domains. A pivotal methodological tool behind this success has been \emph{mean-field analysis}. The history of mean-field analysis, in its current form, goes back to the foundational works of Kurtz~\cite{Kurtz78, Kurtz70, Kurtz71}, Norman~\cite{Norman74, Norman72} and Barbour~\cite{Barbour80}.
The high-level idea is to represent the system state by aggregate Markovian quantities and characterize their rate of change as the system size grows large. 
In the context of load balancing, this representation is the occupancy process $\bq^N(t) = (q_i^N(t))_{i \geq 1}$, where $q_i^N(t)$ denotes the fraction of servers with queue length at least $i$ in a system with $N$ servers at time $t$. As $N\to\infty$, $\bq^N(t)$ tends to behave like a deterministic, continuous system described by an ordinary differential equation (ODE) that is analytically tractable.
A pivotal assumption for the above scheme to work is that the aggregate quantity $\bq^N(t)$ is Markovian such that its rate of change can be expressed as a function of its current state.
If $\bq^N(t)$ is not Markovian, not only does this technique break down, the mean-field approximation may even turn out to be highly inaccurate.

In load balancing systems, if servers are exchangeable, then $\bq^N(t)$ is indeed Markovian. However, the growing heterogeneity in the types of tasks processed by modern data centers has recently motivated the research community to consider systems beyond the exchangeability assumption.
The main reason stems from \emph{data locality}, i.e., the fact that servers need to store resources to process tasks of a particular type locally and have only limited storage space.
Examples of these resources may include databases or machine learning models specific to particular tasks.
This limits the flexibility of the assignment of a task to a queue, which now needs to ensure that the corresponding server is able to process the assigned task.
In fact, the lack of flexibility also arises in much broader contexts such as due to a spatially constrained network architecture (e.g., in
bike-sharing), see~\cite{Gast15, PVBT22, MT00}, or in the context of geographically distributed data centers~\cite{lu2015geographical, maswood2018cost}.
An emerging line of work thus considers a bipartite graph between task types and servers; see for example~\cite{MBL17, WZS20, TX17, RM22, BMW17, CBL19}.
In this \emph{compatibility} graph, an edge between a server and a task type represents the server's ability to process these tasks.
In this model, if the graph is complete bipartite, then the problem reduces to the classical case of a fully flexible system.
In reality, the storage capacity or geographical constraints forces a server to process only a small subset of all task types, leading to sparser network topologies. 
This motivates the study of load balancing in systems with suitably \emph{sparse} bipartite compatibility graphs.\\

\noindent
\textbf{Fundamental barriers.}
The analysis of sparse systems poses significant challenges, mainly due to the fact that the vector $\bq^N(t)$ is no longer Markovian. In fact, for general graphs, there does not even exist a Markovian state descriptor that is an aggregate quantity such as $\bq^N(t)$, and one needs to keep track of the evolution of the entire system in order to know the instantaneous transition rates.
These barriers are the reason, as noted as early as by Mitzenmacher in his thesis~\cite{Mitzenmacher96}, that a network topology is a ``\emph{very interesting question}... (but) \emph{seems to require different techniques}''.
One key question to understand here is: \emph{Under what conditions on the (sparse) compatibility graph does the system behavior retain the performance benefits (in terms of the queue length behavior) of the fully flexible system?}
From a more foundational standpoint, this is equivalent to understanding how much the validity of the mean-field approximation can be extended to non-trivial graphs.

A few recent works have made successful attempts in analyzing compatibility graphs that possess the proper edge-expansion properties~\cite{RM22, MBL17, WZS20}, of which~\cite{RM22} is most relevant to the current work. Here, the JSQ($d$) policy was considered, where each arriving task joins the shortest of $d$ randomly selected compatible queues. The authors showed that if the graph is `well-connected', the limiting occupancy process is indistinguishable from the fully flexible system both in the transient limit and in steady state.
Even though the well-connectedness condition allows the graph to be sparse, it requires the graph to have strong edge-expansion properties in the following sense: 
\emph{Pick any subset of servers of size $\delta N$ for $\delta>0$ however small. Then, asymptotically, almost all task types should be connected to this set and have a $\delta$ fraction of their compatible servers in that set.}
This condition allows the authors in~\cite{RM22} to ensure that, for any occupancy measure, each task type observes approximately the same queue length distribution within their set of compatible servers. As a result, the evolution of the queue length distribution in any neighborhood happens in the same way and this ensures that, asymptotically, the process evolves in the same way as the fully flexible system. 

\begin{figure}
\centering
\begin{subfigure}[b]{0.4\textwidth}
\centering
\includegraphics[width=\textwidth]{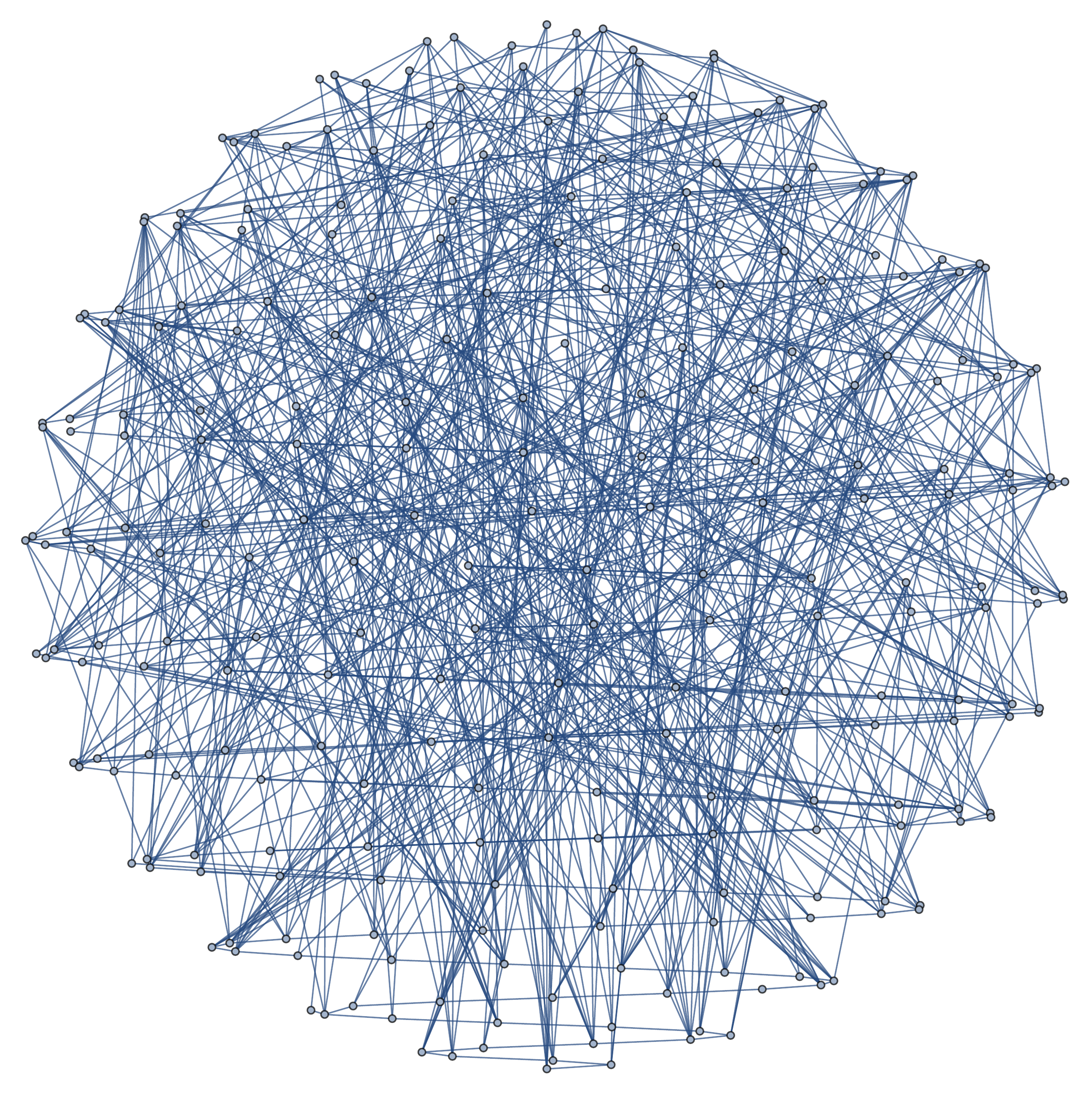}
\caption{Erdős–Rényi graph}
\end{subfigure}\hspace{1cm}
\begin{subfigure}[b]{0.4\textwidth}
\centering
\includegraphics[width=\textwidth]{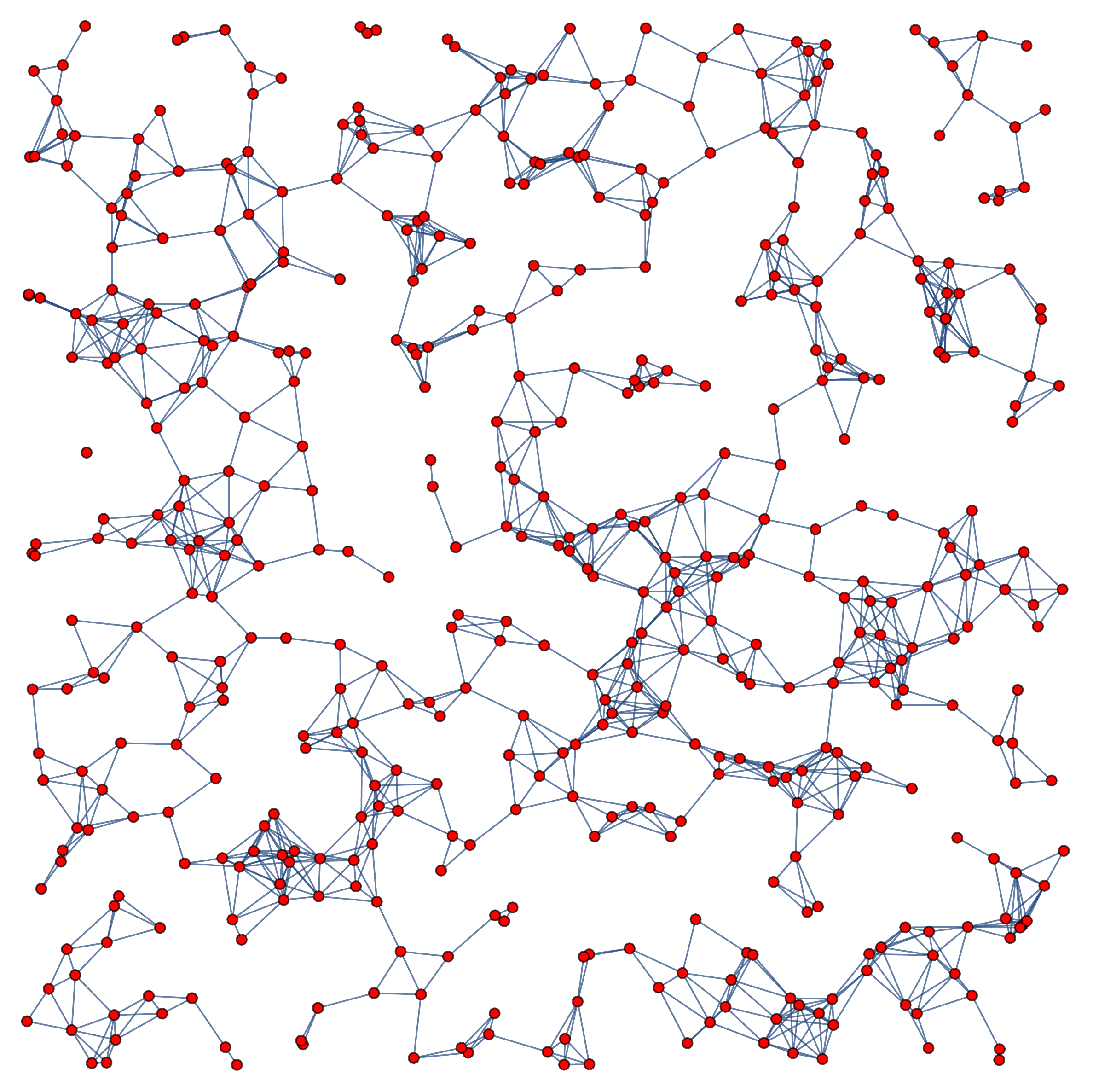}
\caption{Random geometric graph}
\end{subfigure}
\caption{Examples of graph topologies generated by an Erd\H{o}s-R\'enyi graph and a random geometric graph with same average degree $\ln N$, where $N$ is the number of vertices. The picture illustrates the fundamental difference between the nature of global vs.~local connections in the two graphs.}
\end{figure}

The well-connectedness property is not satisfied by spatial graphs such random geometric graphs~\cite{Penrose-book-03}. The edges in a spatial graph are `local', and hence dispatchers in one location cannot assign tasks to servers in spatially distant locations.
However, as already pointed out via numerical simulations in~\cite{RM22}, in steady state, sparse graphs still retain the performance benefits of a fully flexible system, even though the neighborhood coupling based method in~\cite{RM22} fails for these graphs.

Aside from the technical difficulties, there is a fundamental barrier that prevents the mean-field approximation from being applied to spatial graphs. 
This can be understood by a simple counterexample: If all the high queues in the system are located in a small spatial region, then the behavior of the system will be qualitatively different from when they are spread out across the system, as it will take more time for the congestion to disperse throughout the rest of the graph.
In general, in these situations, the behavior of the system in a local neighborhood of the graph may be very different from the global behavior.
Therefore, \emph{one cannot expect the transient behavior of a system with spatial compatibility constraints to coincide with the fully flexible system.}
However, in steady state, it may happen that the situation described in the above counterexample does not occur with high probability, making the steady state still behave like the fully flexible system.
Thus, one needs to characterize the limit of the steady state distribution \emph{without} proceeding via the process-level mean-field limit, as the transient limit will be provably different.
Alternative techniques such as the moment generating function (MGF) method and Lyapunov approaches may allow moment bounds on the steady-state via Stein's method, but cannot commonly be used for the exact characterization of the limit of stationary distributions. 
Although Stein's method has been successfully used for analyzing the join-shortest-queue (JSQ) policy~\cite{WZS20}, these results critically rely on the \emph{state space collapse} or the degeneracy of the steady state, observed as a consequence of JSQ (i.e., all queues are of length zero or one, asymptotically).
When the limit of the stationary distributions is non-degenerate, as is the case in the current paper, we enter uncharted waters in the mean-field approximation literature,
and formalizing a new method to take care of the above difficulties is one of the main contributions of this paper.

\section{Main contributions}\label{sec:main-contrib}

Let $G_N = (V_N, W_N, E_N)$ be a bipartite graph, where $V_N$ denotes the set of servers, $W_N$ denotes the set of task types and $E_N \subseteq V_N \times W_N$ denotes the compatibility constraints. Throughout, we will use the words task-types and dispatchers interchangeably. Here, $N := \lvert V_N \rvert$ equals the number of servers and $M(N) := \lvert W_N \rvert$ equals the number of task types. 
Let $\cN_v := \{ w \in W_N : (v, w) \in E_N \}$ be the compatible task types for a server $v \in V_N$ and $\cN_w := \{ v \in V_N : (v, w) \in  E_N \}$ be the compatible servers for a task type $w \in W_N$. 
Denote $d_v^N = |\cN_v|$ and $d_w^N = |\cN_w|$. 
Tasks of each type arrive as independent Poisson processes of rate $\lambda N/M(N)$ and each task requires an independent and exponentially distributed service time with mean one.
Thus, the total arrival rate is $\lambda N$ and we assume $\lambda<1$ to ensure stability of the system.
If a task arrives at a dispatcher $w \in W_N$, then $d \geq 2$ servers are sampled uniformly at random from $\cN_w$ with replacement, and the task is assigned to the shortest queue among the selected servers, breaking ties at random. 
The tasks in the queue are handled one at a time in first come, first served order.

The criteria for ergodicity of the queue length process for such a system are known and have been developed, for example by Bramson~\cite{Bramson11} and Cardinaels et al.~\cite{CBL19}.
However, in this paper, we work with a slightly stronger, but simplified condition on the graph as follows.
Let 
\begin{equation}\label{eq:stability}
    \rho(G_N) := \max_{v \in V_N} \frac{\lambda N}{M(N)} \sum_{w \in \cN_v} \frac{1}{d_w^N}.
\end{equation}
Using Lyapunov arguments, it is not hard to show that $\rho(G_N)<1$ implies that the queue length process of the system is ergodic for any $d\geq 2$ (Proposition \ref{prop:existence_steadystate}).
Conceptually, $\rho(G_N)$ is the maximum load on a server if each dispatcher uses random routing ($d=1$) and hence it should seem natural that this condition implies stability also for $d\geq 2$.
To avoid heavy-traffic behavior as $N\to\infty$, we will assume that $\rho(G_N)\leq \rho_0$ for all $N\geq 1$ for a constant $\rho_0<1$ throughout.

\begin{remark}\normalfont
In comparison, the stability condition in \cite{CBL19} reduces to: the queue length process is ergodic if for all $w \in W_N$ and $U \subseteq V_N$, there exists a probability distribution $p_{w,U}(\cdot)$ on $U$ such that
\begin{equation}
    \max_{v \in V_N} \frac{\lambda N}{M(N)} \sum_{w \in W_N} \binom{\lvert \cN_w \rvert}{d}^{-1} \sum_{\substack{U \subseteq \cN_w \\ \lvert U \rvert = \min(d, \lvert \cN_w \rvert)}} p_{w,U}(v) < 1.
\end{equation}
However, to prove the mean-field approximation, we require later that $\frac{N}{M(N)} \sum_{w \in \cN_v} \frac{1}{d_w^N} \approx 1$ for all $v \in V_N$ (see the definition of $\gamma(G_N)$ in \eqref{eq:phiandgamma} and Corollary \ref{cor:gen-asymp}) and hence $\rho(G_N) \approx \lambda < 1$ follows immediately. We will therefore work with the simplified stability condition in \eqref{eq:stability}.
\end{remark}

We make contributions on four fronts: (a) We establish bounds on a \emph{large-scale mixing time} of the underlying Markov process; (b) we quantify how much the transient behavior deviates from the mean-field ODE, starting from only empty queues, in terms of certain graph parameters; (c) we combine (a) and (b) to formulate a criterion of when the global quantity $\bq^N(t)$ is asymptotically indistinguishable from the fully flexible system in steady state; and finally
(d) we show how standard generative models for sparse spatial graphs and a large class of sparse regular graphs satisfy this criterion for convergence.\\


\noindent
\textbf{(a)~Large-scale mixing time bounds.}
Mixing time bounds for large-scale systems are known to be hard to obtain. Even without any compatibility constraints, bounding the mixing time for the JSQ($d$) policy for large~$N$ requires significant work~\cite{LM06}.
First, as discussed in~\cite{LM06}, a major challenge is posed by the effect of the starting state. As the state space is infinite, if the system starts from a bad corner of the state space, it may take a very long time to come back to the `regular states', which may even render a mixing time bound useless for our purposes.
Second, in the presence of a compatibility graph structure, regenerative arguments, such as bounding the time the Markov process takes to hit a fixed state~\cite{down1995exponential}, cannot be used either since these regeneration lengths are typically exponential in $N$.
In fact, for large-scale analysis we do not require the conventional notion of mixing time.
Instead, we introduce a notion of \emph{large-scale mixing time} as follows:
starting from a set of suitable states, if we compare the distribution of $\bq^N(t)$ and its steady-state distribution, when can we say that they are `close' in a suitable sense?
Here, it is worth pointing out that, since $\bq^N(t)$ is not a Markov process, by its steady-state distribution we mean the functional $\bq^N(t)$ evaluated on the system in steady state.
We show that this mixing time does not scale with $N$ (Theorem \ref{thm:bound_mixingtime}).
This implies that, starting from the set of suitable states, observing the system  at this mixing time will give us a good approximation of the steady state.
In the above, the set of `suitable states' in particular includes the empty state.
A crucial argument in the proof of Theorem~\ref{thm:bound_mixingtime} relies on a novel stochastic coupling. 
If one copy of the system starts from a state where the queue length at each server is at most the queue length of the corresponding server in another copy of the system, then there exists a stochastic coupling such that this ordering is maintained throughout for any sample path (Proposition \ref{prop:monotone_start}). 
We believe that Proposition \ref{prop:monotone_start} and Theorem \ref{thm:bound_mixingtime} hold for a large class of such monotone systems, which may be of independent interest. \\

\noindent
\textbf{(b)~Process-level limit starting from the empty state.}
As the system quickly converges to the steady-state from any of the set of suitable states, it is sufficient to characterize the sample path of one of these states.
Thus, we next characterize the asymptotics of the sample path of $\bq^N(t)$ starting from a system with only empty queues. Let us introduce two quantities of the underlying graph:
\begin{equation}
\label{eq:phiandgamma}
    \phi(G_N) := \max_{v \in V_N} \left\lvert \frac{N}{M(N)} \sum_{w \in \cN_v} \frac{1}{d_w^N} - 1 \right\rvert\qquad
    \text{ and }\qquad
    \gamma(G_N) := \frac{1}{M(N)} \sum_{w \in W_N} \frac{1}{d_w^N}.
\end{equation}
Loosely speaking, $\phi(G_N)$ quantifies the extent to which the bipartite graph is regular and $\gamma(G_N)$ describes the average inverse degree of the task types.
For example, if $d_w^N = d_{\text{task}}^N$ for all $w \in W_N$ and $d_v^N = d_{\text{server}}^N$ for all $v \in V_N$, then $\phi(G_N) = 0$ and $\gamma(G_N) = 1/d_{\text{task}}^N$ (see also Definition \ref{def:regular}).
We prove that the process-level limit remains close to the system of ODEs for the fully flexible system, in terms of the $\ell_2$-distance, if $\phi(G_N)$ and $\gamma(G_N)$ are suitably small and the system is started in a state that has its queues `\emph{sufficiently spread out}' (Theorem~\ref{thm:process_level_limit}). This in particular includes the state with only empty queues.
Most importantly, \emph{the result in Theorem \ref{thm:process_level_limit} is non-asymptotic.}\\

\noindent
\textbf{(c)~Mean-field approximation.}
Leveraging Theorem \ref{thm:process_level_limit} and the mixing time bound, we determine the applicability of the mean-field approximation for any compatibility graph in terms of the local properties $\phi(G_N)$ and $\gamma(G_N)$. 
In particular, in Theorem \ref{thm:finite_bound} we provide a \emph{finite $N$ guarantee that, for any graph} $G_N$, the $\ell_2$-distance between the steady-state and the fixed point of a system of ODEs is bounded by
\begin{equation}
    \frac{c}{\ln\left( 1 / \max\{ \phi(G_N)^2, \gamma(G_N) \} \right)^\alpha}
\end{equation}
for constants $c, \alpha > 0$ that depend only on $\lambda, \rho_0$, and $d$.
In particular, if $\max\{ \phi(G_N)^2, \gamma(G_N) \}\to 0$ as $N\to\infty$, then
the distribution of $\bq^N(t)$, in steady state, converges weakly to the Dirac delta distribution at the fixed point of the ODE corresponding to the fully flexible system. \\

\noindent
\textbf{(d)~Implications for specific graph classes.}
To show that the conditions on the graph sequence are satisfied by common graphs, we consider 
two sequences of \emph{sparse} graphs for which the condition $\max\{ \phi(G_N)^2, \gamma(G_N) \}\to 0$ as $N\to\infty$ is satisfied.

\textbf{First}, let $(G_N)_{N \geq 1}$ be a sequence of random bipartite geometric graphs. From a high level, these graphs are obtained by placing the dispatchers and the servers at uniformly random locations and connecting a dispatcher and server by an edge if they are at most a fixed distance $r(N)>0$ apart; see Section~\ref{sec:spatialgraph} for a precise definition. 
Recall that $d_v^N$ and $d_w^N$ denote the degree of $v \in V_N$ and $w \in W_N$, respectively.
We prove that, if $r(N)$ is such that $\liminf_{N \to \infty} \E\left[ d_v^N \right] / \ln N = \infty$ and $\liminf_{N \to \infty} \E\left[ d_w^N \right] / \max(\ln M(N), \ln N) = \infty$, then indeed $\max\{ \phi(G_N)^2, \gamma(G_N) \}\to 0$, and $\bq^N(t)$ in steady-state becomes asymptotically indistinguishable from the fully flexible system (Corollary \ref{cor:spatialgraph}). 
Note that these conditions still ensure sparsity in that the degree of a server is nearly a factor $M(N)/\ln N$ smaller as compared to the complete bipartite graph where the degree is $M(N)$.

\textbf{Second}, the above convergence holds in much more generality for a sequence of regular bipartite graphs. That is, $d_v^N$ is the same for all $v$ and $d_w^N$ is the same for all $w$ \emph{within each connected component of the graph}; see Section \ref{sec:regulargraph} for a precise definition.
We prove that the convergence holds whenever $\gamma(G_N) \to 0$, which happens if for example if $\min_{w \in W_N} d_w^N$ diverges (at any rate) as $N \to \infty$ (Corollary \ref{cor:regulargraph}), and thus ensures sparsity. 
This includes \emph{arbitrary deterministic graph sequences} and thus significantly broadens the applicability of the mean-field approximation. 



\section{Main results}

In the following, all graphs will refer to bipartite graphs $G_N = (V_N, W_N, E_N)$ as described in the beginning of Section~\ref{sec:main-contrib}.
We let $X_v(t)$ denote the queue length of a server $v \in V_N$ at time $t$. 
Let $Q_i^N(t) := \sum_{v \in V_N} \mathbbm{1}\left\{ X_v(t) \geq i \right\}$ denote the number of servers with queue length at least $i \in \N$ in the entire system. We will refer to these as \emph{global} quantities. 
The \emph{local} number of servers with queue length at least $i \in \N$, as seen from the perspective of a task type $w \in W_N$, is denoted by $Q_i^{N,w}(t) := \sum_{v \in \cN_w} \mathbbm{1}\left\{ X_v(t) \geq i \right\}$. 
Define their scaled versions as $q_i^N(t) := Q_i^N(t) / N$ and $q_i^{N,w}(t) := Q_i^{N,w}(t) / d_w^N$. 
Note that $\{X_v(t): v\in V_N\}$ is a Markov process, and the vector $(q_i^N(\infty))_{i\geq 1}$ will denote the corresponding steady-state functional of this Markov process.

\subsection{Steady-state approximation for arbitrary graphs}
The JSQ($d$) policy is known for its drastic delay-performance improvement over random routing. 
It is well-known that on a complete bipartite graph with full flexibility, the steady-state quantity $q_i^N(\infty)$ approaches $q_i^*:=\lambda^{\frac{d^i - 1}{d - 1}}$ as $N\to\infty$~\cite{VDK96, Mitzenmacher96}.
This is often referred to as `the power of two effect', meaning that the tail of the queue length distribution decays double-exponentially (in contrast to just exponentially for random routing). 
Recall the definitions of $\phi(G_N)$ and $\gamma(G_N)$ from \eqref{eq:phiandgamma}.
\emph{For an arbitrary compatibility graph $G_N$}, a central result of this paper provides a finite $N$ bound on the 
expected $\ell_2$-distance between $(q_i^N(\infty))_{i\geq 1}$ and $(q_i^*)_{i\geq 1}$:

\begin{theorem}
\label{thm:finite_bound}
Given any $G_N$, if $\rho(G_N) \leq \rho_0 < 1$, then $\boldsymbol{X}(t) = \left( X_v(t) \right)_{v \in V_N}$ is ergodic.
Moreover, if $\max\{ \phi(G_N)^2, \gamma(G_N) \} \leq 1$, then there exist constants $c, \alpha > 0$ (depending only on $\lambda$, $\rho_0$ and $d$) such that
\begin{equation}
    \sum_{i = 1}^\infty \E\left[ \left( q_i^N(\infty) - q_i^* \right)^2 \right]
    \leq \frac{c}{\ln\left( 1 / \max\{ \phi(G_N)^2, \gamma(G_N \} \right)^\alpha},
\end{equation}
where $q_i^* = \lambda^{\frac{d^i - 1}{d - 1}}$ for $i \in \N$.
\end{theorem}
Theorem~\ref{thm:finite_bound} is proved in Section~\ref{ssec:steady-state-proof}.
For large $N$ asymptotics, it also provides a rate of convergence, although we do not expect this rate to be tight for specific sequences of graphs, as the result holds for arbitrary graphs.
The following is an immediate corollary.
\begin{corollary}\label{cor:gen-asymp}
Let $(G_N)_{N\geq 1}$ be a sequence of graphs with $\rho(G_N)\leq \rho_0 < 1$ for all $N \geq 1$ and assume $\max\{ \phi(G_N)^2, \gamma(G_N)\}\to 0$ as $N\to \infty$. Then 
$\sum_{i = 1}^\infty \E\big[ \left( q_i^N(\infty) - q_i^* \right)^2 \big] \to 0$ as $N\to \infty$.
\end{corollary}


\begin{remark}\normalfont
It is worthwhile to note that Theorem~\ref{thm:finite_bound} extends to a bound on any $\ell_p$-distance for $0 < p < \infty$. This follows by bounding the tail sum using Corollary \ref{cor:bound_steadystate} and bounding the finite remainder using H\"older's inequality.
\end{remark}

\subsection{Convergence for specific graph sequences}
\label{sec:spatialgraph}\label{sec:regulargraph}
Let us now discuss two important classes of graph sequences that satisfy the mean-field approximation conditions in Corollary~\ref{cor:gen-asymp}. 
We begin with a popular generative model for spatial graphs.

\begin{definition}[Random bipartite geometric graph]
We say that $G_N$ is a random bipartite geometric graph if $G_N$ is constructed as follows. Let $r(N) > 0$ be fixed and $A$ be the unit torus. 
We assign each $v \in V_N$ and $w \in W_N$ a location $x_v \in A$ and $x_w \in A$, respectively, independently and uniformly at random. Next, $(v, w) \in E_N$ if and only if $\lVert x_v - x_w \rVert_p \leq r(N)$ for $0 < p < \infty$.
\end{definition}

We define random geometric graphs on a torus to avoid boundary effects. From a practical perspective, however, the boundary effects become negligible as $N\to\infty$.
For the following theorem, recall that $d_v^N$ and $d_w^N$ denote the degrees of $v \in V_N$ and $w \in W_N$ in $G_N$, respectively.

\begin{corollary}
\label{cor:spatialgraph}
Let $\left( G_N \right)_{N \geq 1}$ be a sequence of random bipartite geometric graphs, where $r(N)$ is chosen such that
\begin{equation}
    \liminf_{N \to \infty} \frac{\E\left[ d_v^N \right]}{\ln N} = \infty, \quad
    \liminf_{N \to \infty} \frac{\E\left[ d_w^N \right]}{\max(\ln M(N), \ln N)} = \infty.
\end{equation}
Then, almost surely for any realization of the graph sequence $\left( G_N \right)_{N \geq 1}$, $\boldsymbol{X}(t) = \left( X_v(t) \right)_{v \in V}$ is ergodic for all $N$ large enough and $\sum_{i = 1}^\infty \E\big[ \left( q_i^N(\infty) - q_i^* \right)^2 \big] \to 0$ as $N \to \infty$,
where $q_i^* = \lambda^{\frac{d^i - 1}{d - 1}}$ for $i \in \N$.
\end{corollary}

\begin{remark}\normalfont
The reason that `for all $N$ large enough' is added in Corollary~\ref{cor:spatialgraph} is that, for any fixed $N$, with (small but) positive probability, the random graph may not satisfy the stability criterion.
As $N\to\infty$, this probability becomes small and using the Borel-Cantelli lemma, we show that the stability criterion is satisfied almost surely for all $N$ large enough. 
To be precise, the convergence statement for $(q_i^N(\infty))_{i\geq 1}$ should be interpreted for all $N$ large enough where the ergodicity holds.
\end{remark}
The proof relies on verifying the conditions of Corollary~\ref{cor:gen-asymp} using concentration of measure arguments and is given in Section~\ref{ssec:random-verification}.
Next, we consider sequences of regular graphs, in which case, we can allow much more general sequences of graphs.


\begin{definition}[Regular bipartite graph]\label{def:regular}
We say that $G_N$ is a regular bipartite graph if $(v, w) \in E_N$ implies that $N d_v^N = M(N) d_w^N$.
\end{definition}
Note that the definition of a regular bipartite graph implies that the degrees of all servers and all dispatchers are the same \emph{within every connected component}, and it allows the graph to have many connected components.
The next theorem proves the convergence of the steady state for any such regular bipartite graphs with diverging minimum dispatcher degree.

\begin{corollary}
\label{cor:regulargraph}
Let $\left( G_N \right)_{N \geq 1}$ be a sequence of regular bipartite graphs, where $\gamma(G_N) \to 0$ as $N\to\infty$. Then, the queue length process is ergodic for all $N \geq 1$ and $\sum_{i = 1}^\infty \E\big[ \left( q_i^N(\infty) - q_i^* \right)^2 \big] \to 0$ as $N \to \infty$, where $q_i^* = \lambda^{\frac{d^i - 1}{d - 1}}$ for $i \in \N$. The convergence holds in particular if $\min_{w \in W_N} d^N_w\to \infty$ as $N\to\infty$.
\end{corollary}

\begin{proof}
The proof is immediate by observing that, 
due to the regularity of $G_N$, we have 
\begin{equation}
    \phi(G_N)
    := \max_{v \in V_N} \left\lvert \frac{N}{M(N)} \sum_{w \in \cN_v} \frac{1}{d_w^N} - 1 \right\rvert
    = \max_{v \in V_N} \left\lvert \frac{N}{M(N)} \sum_{w \in \cN_v} \frac{M(N)}{N d_v^N} - 1 \right\rvert
    = 0.
\end{equation}
Note also that $\rho(G_N) \leq \lambda (1 + \phi(G_N)) = \lambda < 1$. Therefore, Corollary~\ref{cor:gen-asymp} completes the proof of the first part.
The second part is proved by observing that $\gamma(G_N) \leq  1/(\min_{w \in W_N} d^N_w)$. 
\end{proof}

The rest of the contributions will be pivotal in the proof of Theorem~\ref{thm:finite_bound}.

\subsection{Large-scale mixing-time bound}

A crucial step in identifying the steady-state distribution is to show that the distribution of $\bq^N(t)$ becomes close to its steady state within a large, but finite time. 
We prove that, in appropriate sense, the Markov process mixes in polynomial time, independent of $N$, from any state that is stochastically dominated by the steady-state. 
We use the following notion of stochastic ordering:
\begin{definition}[Stochastic ordering]
For $n \in \N$, let $\bX = (X_1, \ldots, X_n)$ and $Y = (Y_1,\ldots, Y_n)$ be two $n$-dimensional random variables. We write $\bX\leq_{st}\bY$ if there exists a common probability space where $X_i \leq Y_i$ for all $i = 1,\ldots, n$, almost surely.
\end{definition}
To formalize the notion of large-scale mixing time, consider two copies of the system on the same graph $G_N$.
For system $k$, with $k =1,2$, the queue length at server $v\in V_N$ is denoted by $X_v^{(k)}(t)$ and the fraction of servers with queue length at least $i \in \N$ is denoted by $q_i^{N, (k)}(t)$.

\begin{theorem}
\label{thm:bound_mixingtime}
Let $G_N$ be a graph, $\rho(G_N)\leq \rho_0<1$ and $\boldsymbol{X}_0$ be a random variable on $\N^{N}$ such that $\boldsymbol{X}_0\leq_{st} \boldsymbol{X}(\infty)$. 
Suppose $\boldsymbol{X}^{(1)}(0) \overset{d}{=} \boldsymbol{X}_0$ and $\boldsymbol{X}^{(2)}(0) \overset{d}{=} \boldsymbol{X}^{(2)}(\infty)$.
Then there exist a joint probability space and constants $c_1, c_2 > 0$, $0 < \alpha \leq 1$ (depending only on $\rho_0$ and $d$) such that, for all $t \geq 0$,
\begin{equation}
    \sum_{i = 1}^\infty \E\Big[ \Big\lvert q_i^{N,(2)}(t) - q_i^{N,(1)}(t) \Big\rvert \Big] \leq \frac{1}{\left( c_1 + c_2 t \right)^\alpha}.
\end{equation}
\end{theorem}
The proof is given in Section~\ref{sec:proof-mixing} and relies on the fact that the stochastic ordering is maintained throughout for all $t \geq 0$, as shown by the following proposition.

\begin{proposition}
\label{prop:monotone_start}
Under the conditions of Theorem~\ref{thm:bound_mixingtime},
there exists a joint probability space such that $X_v^{(1)}(t) \leq X_v^{(2)}(t)$ for all $v \in V$ and $t \geq 0$, almost surely, along any sample path.
\end{proposition}
The proposition is proved in Section~\ref{sec:proof-mixing}. The proof follows by an induction argument, where we show that the inequality is maintained for each arrival and departure epoch. At an arrival epoch, we use a monotonicity property of the JSQ($d$) policy: if we sample the same $d$ servers in both systems, then the task is routed to a server with a higher queue length in system 2 than in system 1. We relate this behavior to a property of the probabilistic assignment function of JSQ($d$) (Lemma \ref{lem:asg_monotone}). The proposition generalizes to any assignment policy which satisfies such a monotonicity property.

\subsection{Process-level limit the empty state}

As our quantity of interest $\bq^N(t)$ becomes arbitrarily close to the steady-state in finite time, it is sufficient to characterize the transient behavior of one sample path of the system. 
We prove that $\bq^N(t)$ remains close to a system of ODEs if $\phi(G_N)$ and $\gamma(G_N)$ are small (recall \eqref{eq:phiandgamma} for their definition) and the queues in the starting state are `sufficiently spread out'. 

\begin{theorem}
\label{thm:process_level_limit}
Let $G_N$ be a graph, $\rho(G_N) \leq \rho_0$, and $\boldsymbol{\bar{q}}(t) = ( \bar{q}_i(t) )_{i \geq 1}$ be the unique solution to the system of ODEs
\begin{equation}\label{eq:ode}
    \frac{d \bar{q}_i(t)}{dt}
    = \lambda \left( \bar{q}_{i-1}(t)^d - \bar{q}_i(t)^d \right) - \left( \bar{q}_i(t) - \bar{q}_{i+1}(t) \right) \text{ for } i \in \N,
\end{equation}
Then, there exists a constant $c \geq 1$ (depending only on $\rho_0$ and $d$) such that, for all $t \geq 0$,
\begin{equation}
\begin{multlined}
    \E\Big[ \sup_{s \in [0, t]} \sum_{i = 1}^\infty \left( q_i^N(s) - \bar{q}_i(s) \right)^2 \Big]
    \leq 2 \phi(G_N)^2 \Big( \lambda t + \E\Big[ \sum_{i = 1}^\infty q_i^N(0)^2 \Big] \Big) \\
    + 12 e^{c t^2} \Big( t^2 d^2 \phi(G_N)^2 + \E\Big[ \sum_{i = 1}^\infty \Big( \frac{1}{M} \sum_{w \in W} \left\lvert q_i^{N,w}(0) - \bar{q}_i(0) \right\rvert \Big)^2 \Big] + 4 t (\rho_0 d + 1 ) \gamma(G_N) \Big).
\end{multlined}
\end{equation}
\end{theorem}

Note that Theorem~\ref{thm:process_level_limit} is also a non-asymptotic result.  
An immediate corollary is the following.

\begin{corollary}\label{cor:process-all-empty}
Let $(G_N)_{N\geq 1}$ be a sequence of graphs, $\rho(G_N) \leq \rho_0$ and $\max\{\phi(G_N), \gamma(G_N)\}\to 0$ as $N\to \infty$.
Also, assume that $q^N_1(0) = 0$. Then, for any $t\geq 0$,
$$\lim_{N\to \infty}\E\Big[ \sup_{s \in [0, t]} \sum_{i = 1}^\infty \left( q_i^N(s) - \bar{q}_i(s) \right)^2 \Big] = 0,$$
where $(\bar{q}_i(t))_{i \geq 1}$ is as defined in~\eqref{eq:ode}.
\end{corollary}

The `sufficiently spread out' condition in Theorem~\ref{thm:process_level_limit} is imposed by the initial state quantity $\sum_{i = 1}^\infty \big( \frac{1}{M} \sum_{w \in W} \big\lvert q_i^{N,w}(0) - \bar{q}_i(0) \big\rvert \big)^2$. 
This term is small if $\boldsymbol{q}^{N,w}(0) \approx \boldsymbol{\bar{q}}(0)$ for most $w \in W$ and, hence, if the local queue length distribution from the perspective of each task type is approximately equal. 
In particular, this term is zero if the system starts from the empty state, in which case $\boldsymbol{q}^{N,w}(0) = \boldsymbol{\bar{q}}(0) = 0$ for all $w \in W_N$. 
Theorem~\ref{thm:process_level_limit} is proved in Section~\ref{ssec:process-proof}.
The proof relies on tracking a sequence of martingales for each $w \in W_N$ and bounding the $\ell_2$-distance to the ODE by their quadratic variation and quantities such as $\phi(G_N)$ and $\gamma(G_N)$ using Gr\"onwall's inequality. In the proof, the quantity $\phi(G_N)$ is used in \eqref{eq:phi1} and \eqref{eq:phi2} and $\gamma(G_N)$ is used in \eqref{eq:gamma1}.

\begin{remark}\normalfont
One should contrast Theorem~\ref{thm:process_level_limit} and Corollary~\ref{cor:process-all-empty} with the process-level limit result proved in Budhiraja et al.~\cite{BMW17}.
In this paper, the authors considered an undirected version of the model in the current paper.
The model, as is, is not suitable for capturing the task-server compatibility constraints.
An undirected graph would mean that if server~$i$ can process task type~$j$, then server $j$ must be able to process task type~$i$. 
However, a generalization of the model in~\cite{BMW17} to directed graphs can be viewed as a special case of our model: when $M(N) = N$ and there is a perfect matching between the set of servers and the set of dispatchers (equivalently, a dedicated arrival stream per server).
Although the undirected graph assumption is not crucial in~\cite{BMW17}, the $M(N) = N$ assumption plays a major role for the approach to work.
In the current paper, $M(N)$ can grow at any rate (sub-/super-linearly) with $N$.
As a result of the above structural differences, the queue length process in~\cite{BMW17} is ergodic for \emph{any} graph, whereas in our model, this is non-trivial.

Moreover, \cite{BMW17} establishes the process-level convergence if the initial queue lengths at the servers are i.i.d.~from some distribution. 
The idea there is that, if the system starts from a state where the queue lengths at the servers are i.i.d., then any two queue lengths retain their stochastic independence on any finite time interval, asymptotically as $N\to\infty$.
Consequently, the $N$-dimensional queue length vector can be coupled with an infinite-dimensional McKean-Vlasov process where any finite collection of coordinates are independent on any finite time interval.
The assumption that the queue lengths are i.i.d.~at time zero is crucial for this approach to go through. 
As a result, and as already remarked in the conclusion of~\cite{BMW17}, it is unclear how to prove convergence of the steady state. 
In addition to our main contribution on the convergence of steady states, Theorem~\ref{thm:process_level_limit} generalizes the process-level limit beyond the i.i.d.~case. 
It identifies a structural condition on the initial state that ensures the same process-level limit of as the fully flexible system.
\end{remark}

\section{Proofs}
Most of the results in this section are non-asymptotic and hold for any fixed $N$.
Thus, throughout this section, we drop the dependence on $N$ in the notation where possible, for the sake of brevity.

\subsection{Existence of steady-state and moment bound}

We first prove that the Markov process is positive recurrent and has a unique steady-state.

\begin{proposition}
\label{prop:existence_steadystate}
If $\rho(G_N) \leq \rho_0 < 1$, then the Markov process $\boldsymbol{X}(t) = \left( X_v(t) \right)_{v \in V}$ is positive recurrent and there exists a unique steady-state of the process denoted as $\boldsymbol{X}(\infty)$.
\end{proposition}

The proof relies on a Lyapunov argument. Let $V(t) := \sum_{i = 1}^\infty \sum_{j = i}^\infty Q_j(t)$ be the Lyapunov function. If we show that the drift of $V(t)$ is strictly negative anywhere outside of a suitably chosen finite set of states, then this is sufficient for positive recurrence. As such, we compute the drift.

\begin{lemma}
\label{lem:derivative_qsum}
Fix any $i \in \N$ and $t \geq 0$. Then,
\begin{equation}
    \frac{d}{dt} \E\left[ \sum_{j = i}^\infty Q_j(t) \right]
    = \E\left[ \frac{\lambda N}{M} \sum_{w \in W} q_{i-1}^w(t)^d - Q_i(t) \right].
\end{equation}
\end{lemma}

\begin{proof}
To change the value of $\sum_{j = i}^\infty Q_j(t)$, a task must arrive to a server with queue length at least $i-1$ or a task must depart a server with queue length at least $i$.

Let us compute the probability that a task is assigned to a server with queue length at least $i-1$. At the epoch time of an arrival, a task adopts a type $w \in W$ uniformly at random. The task is routed to a server with queue length at least $i-1$ if and only if the system only samples servers with queue length at least $i-1$, which happens with probability $q_{i-1}^w(t-)^d$. This results in a probability of $\frac{1}{M} \sum_{w \in W} q_{i-1}^w(t-)^d$ to be routed to a server with queue length at least $i-1$.

Now, we compute the probability that a task departs a server with queue length at least $i$. At the epoch time of a potential departure, a server $v \in V$ is chosen uniformly at random. A task departs a server with queue length at least $i$ if and only if $v$ has queue length at least $i$. This results in a probability of $q_i(t-)$ to depart a server with queue length at least $i$.

We describe the arrival and departure process as follows. Let $N(t)$ be a Poisson process of rate $(\lambda + 1) N$. An event of the process is either an arrival of type $w \in W$ with probability $\lambda / ((\lambda + 1) M)$ or a potential departure at server $v \in V$ with probability $1 / ((\lambda + 1) N)$, independent of the past. Note that this is equivalent to the model description introduced before. Hence, for any $h > 0$,
\begin{equation}
\begin{split}
    \E\left[ \sum_{j = i}^\infty \Delta Q_j(t) \mid \cF_t \right]
    &= \E\left[ \sum_{j = i}^\infty \Delta Q_j(t) \mid \Delta N(t) = 1, \cF_t \right] \P\left( \Delta N(t) = 1 \right) \\
    &\hspace{3cm}\pm \E\left[ \Delta N(t) \mid \Delta N(t) \geq 2 \right] \P\left( \Delta N(t) \geq 2 \right) \\
   & = \left( \frac{\lambda}{\lambda + 1} \frac{1}{M} \sum_{w \in W} q_{i-1}^w(t)^d - \frac{1}{\lambda + 1} q_i(t) \right) (\lambda + 1) N h e^{-(\lambda + 1) N h} \\
   &\hspace{3cm} \pm \left( (\lambda + 1) N h + 2 \right) \left( (\lambda + 1) N h \right)^2,
\end{split}
\end{equation}
where $\Delta Q_j(t) := Q_j(t + h) - Q_j(t)$ and $\Delta N(t) := N(t + h) - N(t)$. Here, we use the shorthand notation $\pm x$ to denote a term in $[-x, x]$. The equation above implies
\begin{equation}
    \frac{d}{dt} \E\left[ \sum_{j = i}^\infty Q_j(t) \right]
    = \lim_{h \downarrow 0} \frac{\E\left[ \E\left[ \sum_{j = i}^\infty \left( Q_j(t + h) - Q_j(t) \right) \mid \cF_t \right] \right]}{h}
    = \E\left[ \frac{\lambda N}{M} \sum_{w \in W} q_{i-1}^w(t)^d - Q_i(t) \right],
\end{equation}
which completes the proof of the lemma.
\end{proof}

\begin{proof}[Proof of Proposition \ref{prop:existence_steadystate}]
Note that
\begin{equation}
\label{eq:q_bound_global}
    \frac{\lambda N}{M} \sum_{w \in W} q_i^w(t)
    = \frac{\lambda N}{M} \sum_{w \in W} \sum_{\substack{v \in \cN_w \\ X_v(t) \geq i}} \frac{1}{d_w}
    = \sum_{\substack{v \in V \\ X_v(t) \geq i}} \frac{\lambda N}{M} \sum_{w \in \cN_v} \frac{1}{d_w}
    \leq \sum_{\substack{v \in V \\ X_v(t) \geq i}} \rho_0
    = \rho_0 Q_i(t).
\end{equation}
Let $V(t) := \sum_{i = 1}^\infty \sum_{j = i}^\infty Q_j(t)$ and $\boldsymbol{X}(0) = \boldsymbol{x} \in \N^V$. Then, by the monotone convergence theorem,
\begin{equation}
\begin{multlined}
    \frac{d}{dt} \E\left[ V(t) \right]
    = \frac{d}{dt} \sum_{i = 1}^\infty \E\left[ \sum_{j = i}^\infty Q_j(t) \right]
    = \sum_{i = 1}^\infty \frac{d}{dt} \E\left[ \sum_{j = i}^\infty Q_j(t) \right]
    = \sum_{i = 1}^\infty \E\left[ \frac{\lambda N}{M} \sum_{w \in W} q_{i-1}^w(t)^d - Q_i(t) \right] \\
    \leq \sum_{i = 1}^\infty \E\left[ \frac{\lambda N}{M} \sum_{w \in W} q_{i-1}^w(t) - Q_i(t) \right]
    \leq \sum_{i = 1}^\infty \E\left[ \rho_0 Q_{i-1}(t) - Q_i(t) \right]
    = -(1 - \rho_0) \sum_{i = 1}^\infty \E\left[ Q_i(t) \right] + \rho_0 N,
\end{multlined}
\end{equation}
where we use the fact that $\sum_{i = 1}^\infty \E\left[ Q_i(t) \right]$ converges uniformly on any finite interval to exchange derivative and sum in the second equality, Lemma \ref{lem:derivative_qsum} in the third equality and \eqref{eq:q_bound_global} in the second inequality. Let $S := \left\{ \boldsymbol{x} \in \N^V : \sum_{i = 1}^\infty \sum_{v \in V} \mathbbm{1}\{ x_v \geq i \} \leq N / (1 - \rho_0) \right\}$ and note that $S$ is finite. Then,
\begin{equation}
    \frac{d}{dt} \E\left[ V(t) \right] \big\rvert_{t = 0}
    \leq -(1 - \rho_0) \sum_{i = 1}^\infty Q_i(0) + \rho_0 N
    \leq -(1 - \rho_0) N + N \mathbbm{1}\{ \boldsymbol{x} \in S \}.
\end{equation}
Hence, by Theorem 4.2 in \cite{meyn1993stability}, the Markov process $\boldsymbol{X}(t)$ is positive recurrent and there exists a unique steady-state of the process denoted as $\boldsymbol{X}(\infty)$.
\end{proof}

As a consequence, a similar Lyapunov argument shows a moment bound on the steady-state.

\begin{corollary}
\label{cor:bound_steadystate}
If $\rho(G_N) \leq \rho_0 < 1$, then $\E\left[ q_i(\infty) \right] \leq \rho_0^i$ for all $i \in \N$.
\end{corollary}

\begin{proof}
Fix any $i \in \N$ and $t \geq 0$. We let $\boldsymbol{X}(0) \overset{d}{=} \boldsymbol{X}(\infty)$ such that $\boldsymbol{X}(t) \overset{d}{=} \boldsymbol{X}(\infty)$. Then,
\begin{equation}
\begin{multlined}
    0
    = \frac{d}{dt} \E\left[ \sum_{j = i}^\infty Q_j(t) \right]
    = \E\left[ \frac{\lambda N}{M} \sum_{w \in W} q_{i-1}^w(t)^d - Q_i(t) \right] \\
    \leq \E\left[ \frac{\lambda N}{M} \sum_{w \in W} q_{i-1}^w(t) - Q_i(t) \right]
    \leq \E\left[ \rho_0 Q_{i-1}(t) - Q_i(t) \right],
\end{multlined}
\end{equation}
where we use Lemma \ref{lem:derivative_qsum} in the second equality and \eqref{eq:q_bound_global} in the second inequality. Hence, by induction, $\E\left[ q_i(t) \right] \leq \rho_0 \E\left[ q_{i-1}(t) \right] \leq \rho_0^i$, which completes the proof of the lemma.
\end{proof}

\subsection{Proof of the large-scale mixing-time bound}
\label{sec:proof-mixing}

The proofs of this section relies on the following stochastic ordering property of the load balancing  policy as stated in Lemma~\ref{lem:asg_monotone}. 
The lemma is proved in Appendix~\ref{sec:app-lem4}.
We prove this for the JSQ($d$) policy. However, we believe that it is possible to generalize the mixing time bound to a large class of load balancing policies which satisfy an analogous monotonicity property.
\begin{lemma}
\label{lem:asg_monotone}
Fix any $0 \leq y_1 < x_1 \leq 1$ and $0 \leq y_2 < x_2 \leq 1$ such that $x_1 \leq x_2$ and $y_1 \leq y_2$. Then,
\begin{equation}
    \frac{x_1^d - y_1^d}{x_1 - y_1}
    \leq \frac{x_2^d - y_2^d}{x_2 - y_2}.
\end{equation}
\end{lemma}
\begin{proof}[Proof of Proposition~\ref{prop:monotone_start}]
We couple the arrival and potential departure epochs of the two systems such that any arrival of a task type $w \in W$ and any potential departure at a server $v \in V$ happen at the same time in both systems. We proceed to design a stochastic coupling that maintains the inequality $X_v^{(1)}(t) \leq X_v^{(2)}(t)$ for all $v \in V$ on every arrival and potential departure epoch.
At time $t=0$, the inequality is maintained by the stochastic ordering assumption and by defining $\bX^{(1)}(0)$ and $\bX^{(2)}(0)$ on the suitable probability space.

Let $t \geq 0$ be a potential departure epoch at server $v \in V$ and assume that $X_{v'}^{(1)}(t-) \leq X_{v'}^{(2)}(t-)$ for all $v' \in V$. Clearly, $X_v^{(1)}(t) \leq X_v^{(2)}(t)$ also after the departure.

Now, let $t \geq 0$ be an arrival epoch of a task type $w \in W$ and assume that $X_{v'}^{(1)}(t-) \leq X_{v'}^{(2)}(t-)$ for all $v' \in V$. Fix any $v \in \cN_w$ with $i := X_v(t-)$ and let us compute the probability that the task is assigned to $v$. The task is routed to a server with queue length $i$ if and only if the system only samples servers with queue length at least $i$ and not only servers with queue length at least $i + 1$, which happens with probability $q_i^w(t-)^d - q_{i+1}^w(t-)^d$. By symmetry, any server in $\cN_w$ with queue length $i$ has the same probability of receiving the task and there are a total of $Q_i^w(t-) - Q_{i+1}^w(t-)$ of such eligible servers. This results in a probability of
\begin{equation}
    p_v^{(k)} := \frac{q_{X_v^{(k)}(t-)}^{(k),w}(t-)^d - q_{X_v^{(k)}(t-)+1}^{(k),w}(t-)^d}{Q_{X_v^{(k)}(t-)}^{(k),w}(t-) - Q_{X_v^{(k)}(t-)+1}^{(k),w}(t-)},
\end{equation}
of assigning the task to a server $v \in \cN_w$ in system $k = 1,2$. Let $\hat{p}_v := \min\left( p_v^{(1)}, p_v^{(2)} \right)$ be the shared probability mass. For the sake of notation, assume that the servers in $\cN_w$ are ordered and correspond to the integers $\left\{ 1, 2, \dots, d_w \right\}$.  Let $U_t \in [0, 1]$ be a uniform random variable, independent of any other processes and independent across arrival epochs, and which is shared between the two systems. Then, in system $k$, assign the task to server $v \in \cN_w$ if and only if
\begin{equation}
    U_t \in \left[ \sum_{v' = 1}^{v-1} \hat{p}_{v'}, \sum_{v' = 1}^v \hat{p}_{v'} \right) \cup \left[ \sum_{v' \in \cN_w} \hat{p}_{v'} + \sum_{v' = 1}^{v-1} \left( p_{v'}^{(k)} - \hat{p}_{v'} \right), \sum_{v' \in \cN_w} \hat{p}_{v'} + \sum_{v' = 1}^v \left( p_{v'}^{(k)} - \hat{p}_{v'} \right) \right).
\end{equation}
Note that the probability to assign to a server $v \in \cN_w$ in system $k$ is exactly equal to $p_v^{(k)}$.

To verify that the stochastic coupling maintains the ordering of queue lengths, note that if $U_t < \sum_{v' \in \cN_w} \hat{p}_{v'}$, then the task is routed to the same server in both systems by the construction above. Thus, in this case, $X_{v'}^{(1)}(t) \leq X_{v'}^{(2)}(t)$ for all $v' \in V$ also after the arrival.

Next, consider instead that $U_t \geq \sum_{v' \in \cN_w} \hat{p}_{v'}$. Then, the task is routed to two different servers in both systems. Let $v \in \cN_w$ be the server the task is routed to in system 1. Note that it does not matter to which server the task is routed to in system 2, since its queue length will only increase. By the construction above, it must hold that $p_v^{(1)} > \hat{p}_v = p_v^{(2)}$. We claim that this implies that $X_v^{(1)}(t-) < X_v^{(2)}(t-)$. To see why, suppose that $X_v^{(1)}(t-) = X_v^{(2)}(t-)$ instead. Note that
\begin{equation}
\label{eq:q_ordered}
    Q_i^{(1),w}(t-) = \sum_{v' \in \cN_w} \mathbbm{1}\{ X_{v'}^{(1)}(t-) \geq i \} \leq \sum_{v' \in \cN_w} \mathbbm{1}\{ X_{v'}^{(2)}(t-) \geq i \} = Q_i^{(2),w}(t-),
\end{equation}
for all $i \in \N$ and $w \in W$ since $X_{v'}^{(1)}(t-) \leq X_{v'}^{(2)}(t-)$ for all $v' \in V$. Then, by Lemma \ref{lem:asg_monotone},
\begin{equation}
    p_v^{(1)}
    = \frac{1}{d_w} \frac{q_{X_v^{(1)}(t-)}^{(1),w}(t-)^d - q_{X_v^{(1)}(t-)+1}^{(1),w}(t-)^d}{q_{X_v^{(1)}(t-)}^{(1),w}(t-) - q_{X_v^{(1)}(t-)+1}^{(1),w}(t-)}
    \leq \frac{1}{d_w} \frac{q_{X_v^{(2)}(t-)}^{(2),w}(t-)^d - q_{X_v^{(2)}(t-)+1}^{(2),w}(t-)^d}{q_{X_v^{(2)}(t-)}^{(2),w}(t-) - q_{X_v^{(2)}(t-)+1}^{(2),w}(t-)}
    = p_v^{(2)},
\end{equation}
which is a contradiction. Hence, it must be that $X_v^{(1)}(t-) < X_v^{(2)}(t-)$ and therefore $X_{v'}^{(1)}(t) \leq X_{v'}^{(2)}(t)$ for all $v' \in V$ also after the arrival, which completes the proof of the proposition.
\end{proof}

\begin{proof}[Proof of Theorem \ref{thm:bound_mixingtime}]
We couple the two copies of the Markov process according to Proposition \ref{prop:monotone_start} such that $X_v^{(1)}(t) \leq X_v^{(2)}(t)$ for all $v \in V$ and $t \geq 0$, almost surely. This implies that $q_i^{(2),w}(t) \geq q_i^{(1),w}(t)$ for all $w \in W$ and $q_i^{(2)}(t) \geq q_i^{(1)}(t)$ for all $i \in \N$ and $t \geq 0$ by \eqref{eq:q_ordered}. Throughout, we will denote $\Delta_i(t) := q_i^{(2)}(t) - q_i^{(1)}(t)$. Let $\theta := \min\left( 1 / \left( 2 \rho_0 d \right), \rho_0 \right)$ and define $V(t) := \sum_{i = 1}^\infty \theta^i \sum_{j = i}^\infty \Delta_j(t)$. Then, by the monotone convergence theorem,
\begin{equation}
\begin{aligned}
\label{eq:mixingtime_derivative}
    \frac{d}{dt} \E\left[ V(t) \right]
    &= \frac{d}{dt} \sum_{i = 1}^\infty \theta^i \E\left[ \sum_{j = i}^\infty \Delta_j(t) \right]
    = \sum_{i = 1}^\infty \theta^i \frac{d}{dt} \E\left[ \sum_{j = i}^\infty \Delta_j(t) \right] \\
    &= \sum_{i = 1}^\infty \theta^i \E\left[ \frac{\lambda}{M} \sum_{w \in W} \left( q_{i-1}^{(2),w}(t)^d - q_{i-1}^{(1),w}(t)^d \right) - \Delta_i(t) \right] \\
    &\leq \sum_{i = 1}^\infty \theta^i \E\left[ \frac{\lambda d}{M} \sum_{w \in W} \left( q_{i-1}^{(2),w}(t) - q_{i-1}^{(1),w}(t) \right) - \Delta_i(t) \right] \\
    &\leq \sum_{i = 1}^\infty \theta^i \left( \theta \rho_0 d - 1 \right) \E\left[ \Delta_i(t) \right]
    \leq - \frac{1}{2} \sum_{i = 1}^\infty \theta^i \E\left[ \Delta_i(t) \right]
\end{aligned}
\end{equation}
where we use the fact that $\sum_{i = 1}^\infty \theta^i \E\left[ \Delta_i(t) \right]$ converges uniformly since $\theta < 1$ to exchange derivative and sum in the second equality, Lemma \ref{lem:derivative_qsum} in the third equality, the mean value theorem in the first inequality and the definition of $\theta$ in the third inequality. The second inequality follows because
\begin{equation}
\begin{multlined}
    \frac{\lambda N}{M} \sum_{w \in W} \left( q_i^{(2),w}(t) - q_i^{(1),w}(t) \right)
    = \frac{\lambda N}{M} \sum_{w \in W} \sum_{\substack{v \in \cN_w \;:\; X_v^{(2)}(t) \geq i \\ X_v^{(1)}(t) \leq i-1}} \frac{1}{d_w} \\
    = \sum_{\substack{v \in V \;:\; X_v^{(2)}(t) \geq i \\ X_v^{(1)}(t) \leq i-1}} \frac{\lambda N}{M} \sum_{w \in \cN_v} \frac{1}{d_w}
    \leq \sum_{\substack{v \in V \;:\; X_v^{(2)}(t) \geq i \\ X_v^{(1)}(t) \leq i-1}} \rho_0
    = \rho_0 \left( Q_i^{(2)}(t) - Q_i^{(1)}(t) \right).
\end{multlined}
\end{equation}
Next, we find a lower bound on $\sum_{i = 1}^\infty \theta^i \E\left[ \Delta_i(t) \right]$ in terms of $\E\left[ V(t) \right]$. Note that
\begin{equation}
\begin{multlined}
    \E\left[ V(t) \right]
    = \E\left[ \sum_{i = 1}^\infty \sum_{j = i}^\infty \theta^i \Delta_j(t) \right]
    = \E\left[ \sum_{j = 1}^\infty \sum_{i = 1}^j \theta^i \Delta_j(t) \right]
    = \E\left[ \sum_{i = 1}^\infty \frac{\theta \left( 1 - \theta^i \right)}{1 - \theta} \Delta_i(t) \right],
\end{multlined}
\end{equation}
and hence, again by the monotone convergence theorem,
\begin{equation}
\label{eq:mixingtime_boundv}
    \theta \sum_{i = 1}^\infty \E\left[ \Delta_i(t) \right]
    \leq \E\left[ V(t) \right]
    \leq \frac{\theta}{1 - \theta} \sum_{i = 1}^\infty \E\left[ \Delta_i(t) \right].
\end{equation}
Therefore, to find a lower bound on $\sum_{i = 1}^\infty \theta^i \E\left[ \Delta_i(t) \right]$ in terms of $\E\left[ V(t) \right]$, it is sufficient to find a lower bound in terms of $\sum_{i = 1}^\infty \E\left[ \Delta_i(t) \right]$. Let $\eta := \sum_{i = 1}^\infty \E\left[ \Delta_i(t) \right]$. Note that $\E\left[ \Delta_i(t) \right] \leq \E\left[ q_i^{(2)}(t) \right] \leq \rho_0^i$ by Corollary \ref{cor:bound_steadystate}. Thus, a lower bound on $\sum_{i = 1}^\infty \theta^i \E\left[ \Delta_i(t) \right]$ is given by the primal and dual pair
\begin{equation}
\begin{aligned}
    (P) \; \min_{\boldsymbol{x}}& \quad \sum_{i = 1}^\infty \theta^i x_i & (D) \; \max_{z,\boldsymbol{y}}& \quad \eta z - \sum_{i = 1}^\infty \rho_0^i y_i \\
    \text{s.t.}& \quad \begin{aligned}
        \sum_{i = 1}^\infty x_i &= \eta \\
        0 \leq x_i &\leq \rho_0^i
    \end{aligned} & \text{s.t.}& \quad \begin{aligned}
        z - y_i &\leq \theta^i \\
        y_i &\geq 0.
    \end{aligned}
\end{aligned}
\end{equation}
Fix any $i_0 \in \N$. A feasible solution to the dual is $y_i = 0$ for $i < i_0$, $y_i = \theta^{i_0} - \theta^i$ for $i \geq i_0$, and $z = \theta^{i_0}$. As any dual solution provides a lower bound to any primal solution by weak duality, it follows that
\begin{equation}
    \sum_{i = 1}^\infty \theta^i \E\left[ \Delta_i(t) \right]
    \geq \left( \eta -  \sum_{i = i_0}^\infty \rho_0^i \right) \theta^{i_0} + \sum_{i = i_0}^\infty \rho_0^i \theta^i
    = \left( \eta - \frac{\rho_0^{i_0}}{1 - \rho_0} \right) \theta^{i_0} + \frac{\rho_0^{i_0} \theta^{i_0}}{1 - \rho_0 \theta}.
\end{equation}
Now, let $i_0 := \left\lceil \ln(( 1 - \rho_0) \eta) / \ln\left( \rho_0 \right) \right\rceil$. Note that $i_0 \in \N$ since $\eta \leq \rho_0 / (1 - \rho_0)$ and $\rho_0 < 1$. Then,
\begin{equation}
\begin{multlined}
    \sum_{i = 1}^\infty \theta^i \E\left[ \Delta_i(t) \right]
    \geq \left( \eta - \frac{\rho_0^{\ln(( 1 - \rho_0) \eta) / \ln(\rho_0)}}{1 - \rho_0} \right) \theta^{i_0} + \frac{\rho_0^{i_0} \theta^{i_0}}{1 - \rho_0 \theta}
    = \left( \eta - \frac{( 1 - \rho_0) \eta}{1 - \rho_0} \right) \theta^{i_0} + \frac{\rho_0^{i_0} \theta^{i_0}}{1 - \rho_0 \theta} \\
    = \frac{(\rho_0 \theta)^{i_0}}{1 - \rho_0 \theta}
    \geq \frac{\rho_0 \theta}{1 - \rho_0 \theta} \left( \rho_0 \theta \right)^{\ln(( 1 - \rho_0) \eta) / \ln(\rho_0)}
    = \frac{\rho_0 \theta}{1 - \rho_0 \theta} \left( (1 - \rho_0) \eta \right)^{\ln\left( \rho_0 \theta \right) / \ln\left( \rho_0 \right)}.
\end{multlined}
\end{equation}
Let $\alpha := \ln(\theta) / \ln(\rho_0) \geq 1$. The equation above and \eqref{eq:mixingtime_boundv} imply that
\begin{equation}
    \sum_{i = 1}^\infty \theta^i \E\left[ \Delta_i(t) \right]
    \geq \frac{\rho_0 \theta}{1 - \rho_0 \theta} ((1 - \rho_0) \eta)^{1 + \alpha}
    \geq \frac{\rho_0 \theta}{1 - \rho_0 \theta} \left( \frac{(1 - \rho_0) (1 - \theta)}{\theta} \E\left[ V(t) \right] \right)^{1 + \alpha}.
\end{equation}
Thus, we have found a valid lower bound. We apply the lower bound to \eqref{eq:mixingtime_derivative} to find $\frac{d}{dt} \E\left[ V(t) \right] \leq -c_1 \E\left[ V(t) \right]^{1 + \alpha}$, where $c_1 := \rho_0 \theta \left( (1 - \rho_0) (1 - \theta) / \theta \right)^{1 + \alpha} / (2 (1 - \rho_0 \theta)) > 0$. This implies that
\begin{equation}
    \E\left[ V(t) \right]
    \leq \frac{1}{\left( \E\left[ V(0) \right]^{-\alpha} + c_1 \alpha t \right)^{1 / \alpha}}
    \leq \frac{1}{\left( c_2^{-\alpha} + c_1 \alpha t \right)^{1 / \alpha}},
\end{equation}
where $c_2 := \rho_0 \theta / ((1 - \rho_0) (1 - \theta)) > 0$ and we use the fact that $\E\left[ V(0) \right] \leq \theta \sum_{i = 1}^\infty \E\left[ \Delta_i(t) \right] / (1 - \theta) \leq \theta \sum_{i = 1}^\infty \rho_0^i / (1 - \theta) = c_2$ by \eqref{eq:mixingtime_boundv} and Corollary \ref{cor:bound_steadystate} in the second inequality. Hence, by \eqref{eq:mixingtime_boundv},
\begin{equation}
    \sum_{i = 1}^\infty \E\left[ \Delta_i(t) \right]
    \leq \frac{\E\left[ V(t) \right]}{\theta}
    \leq \frac{1}{\theta \left( c_2^{-\alpha} + c_1 \alpha t \right)^{1 / \alpha}},
\end{equation}
which completes the proof of the theorem.
\end{proof}

\subsection{Proof of the process-level limit}
\label{ssec:process-proof}

\begin{lemma}
\label{lem:q_martingale}
Fix any $i \in \N$ and $w \in W$. The process
\begin{equation}
    M_i^w(t)
    := Q_i^w(t) - Q_i^w(0) - \int_0^t \frac{\lambda N}{M} \sum_{\substack{v \in \cN_w \\ X_v(s) = i-1}} \sum_{w' \in \cN_v} \frac{q_{i-1}^{w'}(s)^d - q_i^{w'}(s)^d}{Q_{i-1}^{w'}(s) - Q_i^{w'}(s)} - \left( Q_i^w(s) - Q_{i+1}^w(s) \right) \diff s,
\end{equation}
is a square-integrable martingale started at zero. Moreover, the quadratic variation $\left[ M_i^w \right]_t$ satisfies
\begin{equation}
    \E\left[ \left[ M_i^w \right]_t \right]
    = \E\left[ \int_0^t \frac{\lambda N}{M} \sum_{\substack{v \in \cN_w \\ X_v(s) = i-1}} \sum_{w' \in \cN_v} \frac{q_{i-1}^{w'}(s)^d - q_i^{w'}(s)^d}{Q_{i-1}^{w'}(s) - Q_i^{w'}(s)} + \left( Q_i^w(s) - Q_{i+1}^w(s) \right) \diff s \right].
\end{equation}
\end{lemma}

The proof of Lemma \ref{lem:q_martingale} is provided in Appendix \ref{sec:q_martingale}.

\begin{proof}[Proof of Theorem \ref{thm:process_level_limit}]
Fix any $i \in \N$, $w \in W$ and $t \geq 0$ and let $d_i^w(t) := \lvert q_i^w(t) - \bar{q}_i(t) \rvert$. Then,
\begin{equation}
\begin{multlined}
\label{eq:d_bound}
    d_i^w(t)
    \leq \lambda \int_0^t \Big\lvert \frac{1}{d_w} \frac{N}{M} \sum_{\substack{v \in \cN_w \\ X_v(s) = i-1}} \sum_{w' \in \cN_v} \frac{q_{i-1}^{w'}(s)^d - q_i^{w'}(s)^d}{Q_{i-1}^{w'}(s) - Q_i^{w'}(s)} - \left( \bar{q}_{i-1}(s)^d - \bar{q}_i(s)^d \right) \Big\rvert \diff s \\ + \int_0^t \left\lvert \left( q_i^w(s) - q_{i+1}^w(s) \right) - \left( \bar{q}_i(s) - \bar{q}_{i+1}(s) \right) \right\rvert \diff s + d_i^w(0) + \frac{\lvert M_i^w(t) \rvert}{d_w},
\end{multlined}
\end{equation}
where $M_i^w(t)$ is a square-integrable martingale as defined in Lemma \ref{lem:q_martingale}. We proceed by bounding the terms on the right-hand side. The term in the first integral in \eqref{eq:d_bound} is upper bounded by
\begin{equation}
\begin{multlined}
    \Big\lvert \frac{1}{d_w} \frac{N}{M} \sum_{\substack{v \in \cN_w \\ X_v(s) = i-1}} \sum_{w' \in \cN_v} \frac{q_{i-1}^{w'}(s)^d - q_i^{w'}(s)^d}{Q_{i-1}^{w'}(s) - Q_i^{w'}(s)} - \left( \bar{q}_{i-1}(s)^d - \bar{q}_i(s)^d \right) \Big\rvert \\
    \leq \Big\lvert \frac{1}{d_w} \sum_{\substack{v \in \cN_w \\ X_v(s) = i-1}} \sum_{w' \in \cN_v} \frac{N}{M} \frac{q_{i-1}^{w'}(s)^d - q_i^{w'}(s)^d}{Q_{i-1}^{w'}(s) - Q_i^{w'}(s)} - \left( q_{i-1}^w(s) - q_i^w(s) \right) \frac{\bar{q}_{i-1}(s)^d - \bar{q}_i(s)^d}{\bar{q}_{i-1}(s) - \bar{q}_i(s)} \Big\rvert \\ + \left\lvert \left( q_{i-1}^w(s) - q_i^w(s) \right) - \left( \bar{q}_{i-1}(s) - \bar{q}_i(s) \right) \right\rvert \frac{\bar{q}_{i-1}(s)^d - \bar{q}_i(s)^d}{\bar{q}_{i-1}(s) - \bar{q}_i(s)} \\
    \leq \frac{1}{d_w} \sum_{\substack{v \in \cN_w \\ X_v(s) = i-1}} \Big\lvert \sum_{w' \in \cN_v} \frac{N}{M} \frac{q_{i-1}^{w'}(s)^d - q_i^{w'}(s)^d}{Q_{i-1}^{w'}(s) - Q_i^{w'}(s)} - \frac{\bar{q}_{i-1}(s)^d - \bar{q}_i(s)^d}{\bar{q}_{i-1}(s) - \bar{q}_i(s)} \Big\rvert
    + d \left( d_{i-1}^w(s) + d_i^w(s) \right),
\end{multlined}
\end{equation}
where we use the triangle inequality in the first inequality and the mean value theorem in the second inequality. The first term on the right-hand side above is further upper bounded by
\begin{equation}
\begin{multlined}
\label{eq:phi1}
    \frac{1}{d_w} \sum_{\substack{v \in \cN_w \\ X_v(s) = i-1}} \left\lvert \sum_{w' \in \cN_v} \frac{N}{M} \frac{q_{i-1}^{w'}(s)^d - q_i^{w'}(s)^d}{Q_{i-1}^{w'}(s) - Q_i^{w'}(s)} - \frac{\bar{q}_{i-1}(s)^d - \bar{q}_i(s)^d}{\bar{q}_{i-1}(s) - \bar{q}_i(s)} \right\rvert \\
    \leq \frac{1}{d_w} \sum_{\substack{v \in \cN_w \\ X_v(s) = i-1}} \frac{N}{M} \sum_{w' \in \cN_v} \frac{1}{d_{w'}} \left\lvert  \frac{q_{i-1}^{w'}(s)^d - q_i^{w'}(s)^d}{q_{i-1}^{w'}(s) - q_i^{w'}(s)} - \frac{\bar{q}_{i-1}(s)^d - \bar{q}_i(s)^d}{\bar{q}_{i-1}(s) - \bar{q}_i(s)} \right\rvert \\
    + \frac{1}{d_w} \sum_{\substack{v \in \cN_w \\ X_v(s) = i-1}} \left\lvert \frac{N}{M} \sum_{w' \in \cN_v} \frac{1}{d_{w'}} - 1 \right\rvert \frac{\bar{q}_{i-1}(s)^d - \bar{q}_i(s)^d}{\bar{q}_{i-1}(s) - \bar{q}_i(s)} \\
    \leq \frac{1}{d_w} \sum_{v \in \cN_w} \frac{N}{M} \sum_{w' \in \cN_v} \frac{K}{d_{w'}} \left( d_{i-1}^{w'}(s) + d_i^{w'}(s) \right)
    + d \phi(G) \left( q_{i-1}^w(s) - q_i^w(s) \right),
\end{multlined}
\end{equation}
where we use the triangle inequality in the first inequality and Lemma \ref{lem:asg_lipschitz} and the mean value theorem in the second inequality. Then, summing the first term on the right-hand side over $w \in W$,
\begin{equation}
\begin{multlined}
    \sum_{w \in W} \frac{1}{d_w} \sum_{v \in \cN_w} \frac{N}{M} \sum_{w' \in \cN_v} \frac{K}{d_{w'}} \left( d_{i-1}^{w'}(s) + d_i^{w'}(s) \right)
    = \sum_{w \in W} \frac{1}{d_w} \sum_{w' \in W} \frac{N}{M} \sum_{v \in \cN_w \cap \cN_{w'}} \frac{K}{d_{w'}} \left( d_{i-1}^{w'}(s) + d_i^{w'}(s) \right) \\
    = \sum_{w' \in W} \frac{1}{d_{w'}} \sum_{w \in W} \frac{N}{M} \sum_{v \in \cN_w \cap \cN_{w'}} \frac{K}{d_w} \left( d_{i-1}^{w'}(s) + d_i^{w'}(s) \right)
    = \sum_{w' \in W} \frac{1}{d_{w'}} \sum_{v \in \cN_{w'}} \frac{N}{M} \sum_{w \in \cN_v} \frac{K}{d_w} \left( d_{i-1}^{w'}(s) + d_i^{w'}(s) \right) \\
    \leq \sum_{w' \in W} \frac{1}{d_{w'}} \sum_{v \in \cN_{w'}} \frac{\rho_0 K}{\lambda} \left( d_{i-1}^{w'}(s) + d_i^{w'}(s) \right)
    = \sum_{w' \in W} \frac{\rho_0 K}{\lambda} \left( d_{i-1}^{w'}(s) + d_i^{w'}(s) \right).
\end{multlined}
\end{equation}
The term in the second integral in \eqref{eq:d_bound} is bounded by
\begin{equation}
\begin{multlined}
    \left\lvert \left( q_i^w(s) - q_{i+1}^w(s) \right) - \left( \bar{q}_i(s) - \bar{q}_{i+1}(s) \right) \right\rvert
    \leq d_i^w(s) + d_{i+1}^w(s).
\end{multlined}
\end{equation}
Therefore, putting the above together, by Jensen's inequality and the Cauchy-Schwartz inequality,
\begin{equation}
\begin{multlined}
    \left( \sum_{w \in W} d_i^w(t) \right)^2
    \leq \Bigg( \sum_{w \in W} \bigg( \int_0^t (\rho_0 K + d) \left( d_{i-1}^w(s) + d_i^w(s) \right) + d \phi(G) \left( q_{i-1}^w(s) - q_i^w(s) \right) \\
    + \left( d_i^w(s) + d_{i+1}^w(s) \right) \diff s + d_i^w(0) + \frac{\lvert M_i^w(t) \rvert}{d_w} \bigg) \Bigg)^2 \\
    \leq 6 \left( \int_0^t \sum_{w \in W} c_1 d_{i-1}^w(s) \diff s \right)^2 + 6 \left( \int_0^t \sum_{w \in W} c_1 d_i^w(s) \diff s \right)^2 + 6 \left( \int_0^t \sum_{w \in W} d_{i+1}^w(s) \diff s \right)^2 \\
    + 6 \left( \int_0^t \sum_{w \in W} d \phi(G) \left( q_{i-1}^w(s) - q_i^w(s) \right) \diff s \right)^2 + 6 \left( \sum_{w \in W} d_i^w(0) \right)^2 + 6 \left( \sum_{w \in W} \frac{\lvert M_i^w(t) \rvert}{d_w} \right)^2 \\
    \leq 6t c_1^2 \int_0^t \left( \sum_{w \in W} d_{i-1}^w(s) \right)^2 \diff s + 6t c_1^2 \int_0^t \left( \sum_{w \in W} d_i^w(s) \right)^2 \diff s + 6t \int_0^t \left( \sum_{w \in W} d_{i+1}^w(s) \right)^2 \diff s \\
    + 6 M t d^2 \phi(G)^2 \int_0^t \sum_{w \in W} \left( q_{i-1}^w(s) - q_i^w(s) \right)^2 \diff s + 6 \left( \sum_{w \in W} d_i^w(0) \right)^2 + 6 M \sum_{w \in W} \frac{M_i^w(t)^2}{d_w^2},
\end{multlined}
\end{equation}
where $c_1 := \rho_0 K + d + 1$. Then, by the monotone convergence theorem,
\begin{equation}
\begin{multlined}
    \sup_{s \in [0, t]} \sum_{i = 1}^\infty \left( \frac{1}{M} \sum_{w \in W} d_i^w(s) \right)^2
    \leq c_2 t \int_0^t \sup_{u \in [0, s]} \sum_{i = 1}^\infty \left( \frac{1}{M} \sum_{w \in W} d_i^w(u) \right)^2 \diff s \\
    + 6 t^2 d^2 \phi(G)^2
    + 6 \sum_{i = 1}^\infty \left( \frac{1}{M} \sum_{w \in W} d_i^w(0) \right)^2
    + \frac{6}{M} \sum_{i = 1}^\infty \sum_{w \in W} \sup_{s \in [0, t]} \frac{M_i^w(s)^2}{d_w^2},
\end{multlined}
\end{equation}
where $c_2 := 6 \left( 2 c_1^2 + 1 \right)$. Hence, by Gr\"onwall's inequality,
\begin{equation}
\label{eq:diw_bounded}
    \sup_{s \in [0, t]} \sum_{i = 1}^\infty \left( \frac{1}{M} \sum_{w \in W} d_i^w(s) \right)^2
    \leq 6 e^{c_2 t^2} \left( t^2 d^2 \phi(G)^2 + \sum_{i = 1}^\infty \left( \frac{1}{M} \sum_{w \in W} d_i^w(0) \right)^2 + \sum_{i = 1}^\infty \frac{1}{M} \sum_{w \in W} \sup_{s \in [0, t]} \frac{M_i^w(s)^2}{d_w^2} \right).
\end{equation}
This almost completes the proof of the theorem. Now, by Jensen's inequality,
\begin{equation}
\begin{multlined}
\label{eq:di_bounded}
    \sum_{i = 1}^\infty \left( q_i(t) - \bar{q}_i(t) \right)^2
    \leq 2 \sum_{i = 1}^\infty \left( \frac{1}{M} \sum_{w \in W} q_i^w(t) - q_i(t) \right)^2 + 2 \sum_{i = 1}^\infty \left( \frac{1}{M} \sum_{w \in W} q_i^w(t) - \bar{q}_i(s) \right)^2 \\
    \leq 2 \phi(G)^2 \sum_{i = 1}^\infty q_i(t)^2 + 2 \sum_{i = 1}^\infty \left( \frac{1}{M} \sum_{w \in W} d_i^w(t) \right)^2 \\
    \leq 2 \phi(G)^2 \left( \frac{N_a(t)}{N} + \sum_{i = 1}^\infty q_i(0)^2 \right) + 2 \sum_{i = 1}^\infty \left( \frac{1}{M} \sum_{w \in W} d_i^w(t) \right)^2,
\end{multlined}
\end{equation}
where $N_a(t)$ denotes the number of arrivals until time $t$ and the second inequality follows because
\begin{equation}
\begin{multlined}
\label{eq:phi2}
    \left\lvert \frac{N}{M} \sum_{w \in W} q_i^w(t) - Q_i(t) \right\rvert
    = \left\lvert \frac{N}{M} \sum_{w \in W} \sum_{\substack{v \in \cN_w \\ X_v(t) \geq i}} \frac{1}{d_w} - Q_i(t) \right\rvert
    = \left\lvert \sum_{\substack{v \in V \\ X_v(t) \geq i}} \frac{N}{M} \sum_{w \in \cN_v} \frac{1}{d_w} - \sum_{\substack{v \in V \\ X_v(t) \geq i}} 1 \right\rvert \\
    \leq \sum_{\substack{v \in V \\ X_v(t) \geq i}} \left\lvert \frac{N}{M} \sum_{w \in \cN_v} \frac{1}{d_w} - 1 \right\rvert
    \leq \sum_{\substack{v \in V \\ X_v(t) \geq i}} \phi(G)
    = \phi(G) Q_i(t).
\end{multlined}
\end{equation}
Then, \eqref{eq:diw_bounded} and \eqref{eq:di_bounded} together imply that
\begin{equation}
\begin{multlined}
\label{eq:gamma1}
    \E\left[ \sup_{s \in [0, t]} \sum_{i = 1}^\infty \left( q_i(s) - \bar{q}_i(s) \right)^2 \right]
    \leq 2 \phi(G)^2 \E\left[ \frac{N_a(t)}{N} + \sum_{i = 1}^\infty q_i(0)^2 \right] + 2 \E\left[ \sup_{s \in [0, t]} \sum_{i = 1}^\infty \left( \frac{1}{M} \sum_{w \in W} d_i^w(s) \right)^2 \right] \\
    \leq 2 \phi(G)^2 \left( \lambda t + \E\left[ \sum_{i = 1}^\infty q_i(0)^2 \right] \right) \\
    + 12 e^{c_2 t^2} \left( t^2 d^2 \phi(G)^2 + \E\left[ \sum_{i = 1}^\infty \left( \frac{1}{M} \sum_{w \in W} d_i^w(0) \right)^2 \right] + 4 t (\rho_0 d + 1 ) \gamma(G) \right),
\end{multlined}
\end{equation}
where we use Lemma \ref{lem:martingale_vanishes} in the second inequality. This completes the proof of the theorem.
\end{proof}

\subsection{Analysis of the steady-state}\label{ssec:steady-state-proof}

The mixing time bound above shows that the system is close to the steady-state at a large, but finite time, starting from the empty state. The process-level limit characterizes this sample path and proves that the system remains close to an ODE. Together with a standard global convergence result, this implies that the steady-state is close to the fixed point of the system of ODEs.

\begin{proof}[Proof of Theorem \ref{thm:finite_bound}]
Proposition \ref{prop:existence_steadystate} proves the first half of the theorem. To prove the second half, let $\boldsymbol{\bar{q}}(t)$ be the unique solution to the ODEs in Theorem \ref{thm:process_level_limit}, where $\boldsymbol{\bar{q}}(0) = 0$. Theorem 3.6 in \cite{Mitzenmacher01} shows that there exist constants $c_1, c_2 > 0$ (depending only on $\lambda$) such that
\begin{equation}
    \sum_{i = 1}^\infty \left( \bar{q}_i(t) - q_i^* \right)^2
    \leq \sum_{i = 1}^\infty \left\lvert \bar{q}_i(t) - q_i^* \right\rvert
    \leq c_1 e^{-c_2 t}
    \leq \frac{c_1}{1 + c_2 t}.
\end{equation}
Throughout, denote $\eta = \max\{ \phi(G)^2, \gamma(G) \}$. Let $\boldsymbol{X}^{(1)}(t)$ and $\boldsymbol{X}^{(2)}(t)$ be two copies of the Markov process, where $\boldsymbol{X}^{(1)}(0) = 0$ and $\boldsymbol{X}^{(2)}(0) \overset{d}{=} \boldsymbol{X}^{(2)}(\infty)$. Then, there exist constants $c_4, c_5, c_1', c_2', c_4' > 0$, $c_3 \geq 1$ and $0 < \alpha \leq 1$ (depending only on $\lambda$, $\rho_0$ and $d$) such that, for all $t \geq 1$,
\begin{equation}
\begin{multlined}
    \sum_{i = 1}^\infty \E\left[ \left( q_i^{(2)}(\infty) - q_i^* \right)^2 \right]
    = \sum_{i = 1}^\infty \E\left[ \left( q_i^{(2)}(t) - q_i^* \right)^2 \right] \\
    \leq 3 \sum_{i = 1}^\infty \E\left[ \left( q_i^{(2)}(t) - q_i^{(1)}(t) \right)^2 \right]
    + 3 \sum_{i = 1}^\infty \E\left[ \left( q_i^{(1)}(t) - \bar{q}_i(t) \right)^2 \right]
    + 3 \sum_{i = 1}^\infty \E\left[ \left( \bar{q}_i(t) - q_i^* \right)^2 \right] \\
    \leq \frac{3}{\left( c_4 + c_5 t \right)^\alpha} + 6 \phi(G)^2 \lambda t + 36 e^{c_3 t^2} \left( t^2 d^2 \phi(G)^2 + 4 t (\rho_0 d + 1 ) \gamma(G) \right) + \frac{3 c_1}{1 + c_2 t} \\
    \leq \frac{1}{\left( c_1' + c_2' t \right)^\alpha} + c_4' t^2 e^{c_3 t^2} \eta,
\end{multlined}
\end{equation}
where we use Jensen's inequality in the first inequality and Theorem \ref{thm:bound_mixingtime} and \ref{thm:process_level_limit} in the second inequality. We consider two cases. If $\ln\left(1 / \eta \right) \geq 2 c_3$, then let $t = \sqrt{\ln\left(1 / \eta \right) / (2 c_3)} \geq 1$ such that
\begin{equation}
    t^2 e^{c_3 t^2} \eta
    = \frac{\ln\left( 1 / \eta \right) \sqrt{\eta}}{2 c_3}
    \leq \frac{1}{c_3 \sqrt{\ln\left( 1 / \eta \right)}}
    \leq \frac{1}{\left( c_3 \sqrt{\ln\left( 1 / \eta \right)} \right)^\alpha}
\end{equation}
where we use that $\ln(1 / x) \sqrt{x} \leq 2 / \sqrt{\ln(1 / x)}$ for $0 \leq x \leq 1$ in the first inequality. Therefore,
\begin{equation}
    \sum_{i = 1}^\infty \E\left[ \left( q_i^{(2)}(\infty) - q_i^* \right)^2 \right]
    \leq \frac{1}{\left( c_1' + c_2' \sqrt{\ln\left(1 / \eta \right)} \right)^\alpha} + \frac{c_4'}{\left( c_3 \sqrt{\ln\left(1 / \eta \right)} \right)^\alpha}
    \leq \frac{c_3^\alpha + c_2^{\prime \alpha} c_4'}{\left( c_2' c_3 \sqrt{\ln\left(1 / \eta \right)} \right)^\alpha}.
\end{equation}
If instead $\ln\left(1 / \eta \right) < 2 c_3$, then let $t = 1$ such that
\begin{equation}
\begin{multlined}
    \sum_{i = 1}^\infty \E\left[ \left( q_i^{(2)}(\infty) - q_i^* \right)^2 \right]
    \leq \frac{1}{\left( c_1' + c_2' \right)^\alpha} + c_4' e^{c_3} \eta \\
    \leq \frac{\sqrt{2 c_3}}{\left( c_1' + c_2' \right)^\alpha \sqrt{\ln\left(1 / \eta \right)}} + \frac{c_4' e^{c_3}}{\sqrt{\ln\left(1 / \eta \right)}}
    \leq \frac{\sqrt{2 c_3} + c_4' e^{c_3} (c_1' + c_2')^\alpha}{\left( c_1' + c_2' \right)^\alpha \sqrt{\ln\left(1 / \eta \right)}},
\end{multlined}
\end{equation}
where we use that $x \leq 1 / \sqrt{\ln(1 / x)}$ for $0 \leq x \leq 1$ in the second inequality. This completes the proof of the theorem.
\end{proof}

\subsection{Verification for random bipartite geometric graphs}
\label{ssec:random-verification}

\begin{proof}[Proof of Corollary~\ref{cor:spatialgraph}]
Fix any $v \in V_N$ and $0 < \varepsilon \leq 1 / 2$. As each $w \in W_N$ is placed independently and uniformly at random, $d_v^N$ is distributed as a binomial random variable. Therefore, a Chernoff bound (see e.g. Corollary 2.3 in \cite{JLR00}) shows that
\begin{equation}
    \P\left( \left\lvert d_v^N - \E\left[ d_v^N \right] \right\rvert \geq \varepsilon \E\left[ d_v^N \right] \right)
    \leq 2 \exp\left( -\varepsilon^2 \E\left[ d_v^N \right] / 3 \right).
\end{equation}
A similar Chernoff bound holds for $w \in W_N$. Let $E_N$ denote the event that there exists a $v \in V_N$ such that $\left\lvert d_v^N - \E\left[ d_v^N \right] \right\rvert \geq \varepsilon \E\left[ d_v^N \right]$ or there exists a $w \in W_N$ such that $\left\lvert d_w^N - \E\left[ d_w^N \right] \right\rvert \geq \varepsilon \E\left[ d_w^N \right]$. Let $N_1$ be large enough such that $\varepsilon^2 \E\left[ d_v^N \right] / 3 \geq 3 \ln N$, $\varepsilon^2 \E\left[ d_w^N \right] / 3 \geq 3 \ln N$ and $\varepsilon^2 \E\left[ d_w^N \right] / 3 \geq 3 \ln M$ for all $N \geq N_1$. Then,
\begin{equation}
\begin{multlined}
    \P\left( E_N \right)
    \leq \sum_{v \in V_N} \P\left( \left\lvert d_v^N - \E\left[ d_v^N  \right] \right\rvert \geq \varepsilon \E\left[ d_v^N \right] \right) + \sum_{w \in W_N} \P\left( \left\lvert d_w^N - \E\left[ d_w^N \right] \right\rvert \geq \varepsilon \E\left[ d_w^N \right] \right) \\
    \leq 2 N \exp\left( -\varepsilon^2 \E\left[ d_v^N \right] / 3 \right) + 2 M \exp\left( -\varepsilon^2 \E\left[ d_w^N \right] / 3 \right) \\
    \leq 2 N \exp\left( -3 \ln N \right) + 2 M \exp\left( - \ln(M) - 2 \ln(N) \right)
    = \frac{4}{N^2},
\end{multlined}
\end{equation}
for all $N \geq N_1$. Hence, $\sum_{N = 1}^\infty \P\left( E_N \right) < \infty$ and the Borel-Cantelli lemma shows that, almost surely, there exists $N_2 < \infty$ such that $E_N$ does not occur for all $N \geq N_2$. This implies in particular that,
\begin{equation}
\begin{split}
    \frac{1 - \varepsilon}{1 + \varepsilon}
    = \frac{N}{M(N)} \frac{(1 - \varepsilon) \E\left[ d_v^N \right]}{(1 + \varepsilon) \E\left[ d_w^N \right]} 
   & \leq \frac{N}{M(N)} \frac{\min_{v \in V_N} d_v^N}{\max_{w \in W_N} d_w^N}
    \leq \frac{N}{M(N)} \sum_{w \in \cN_v} \frac{1}{d_w^N}\\
    &\leq \frac{N}{M(N)} \frac{\max_{v \in V_N} d_v^N}{\min_{w \in W_N} d_w^N} 
    \leq \frac{N}{M(N)} \frac{(1 + \varepsilon) \E\left[ d_v^N \right]}{(1 - \varepsilon) \E\left[ d_w^N \right]}
    = \frac{1 + \varepsilon}{1 - \varepsilon},
\end{split}
\end{equation}
for all $N \geq N_2$ and therefore
\begin{equation}
    \phi(G_N)
    := \max_{v \in V_N} \left\lvert \frac{N}{M(N)} \sum_{w \in W_N} \frac{1}{d_w^N} - 1 \right\rvert
    \leq \max\left( 1 - \frac{1 - \varepsilon}{1 + \varepsilon}, \frac{1 + \varepsilon}{1 - \varepsilon} - 1 \right)
    \leq \frac{2 \varepsilon}{1 - \varepsilon}
    \leq 4 \varepsilon,
\end{equation}
for all $N \geq N_2$. Also,
\begin{equation}
    \gamma(G_N)
    := \frac{1}{M(N)} \sum_{w \in W_N} \frac{1}{d_w^N}
    \leq \frac{1}{\min_{w \in W_N} d_w^N}
    \leq \frac{1}{(1 - \varepsilon) \E\left[ d_w^N \right]}
    \leq \frac{2}{\ln N},
\end{equation}
for all $N \geq N_2$. Note also that $\rho(G_N) \leq \lambda (1 + \phi(G_N)) \leq \lambda (1 + 4 \varepsilon) < 1$ for all $N \geq N_2$ and $\varepsilon$ small enough. Therefore, Theorem \ref{thm:finite_bound} completes the proof.
\end{proof}

\section{Numerical experiments}

We perform numerical experiments to complement the theoretical results. The experiments are in the scenario where $M(N) = N$ for simplicity. We simulate two types of graph sequences: random bipartite geometric graphs and random regular bipartite graphs. The random geometric graph is generated as described in its definition in Section~\ref{sec:spatialgraph}. The random regular bipartite graph is generated by fixing a degree $k$ upfront. Then, $k$ half-edges are created at each server $v \in V_N$ and task-type $w \in W_N$. The half-edges at the servers are connected to the half-edges at the task types by sequentially picking two available half-edges at random, one at the server side and one at the task-type side. and creating an edge between them. 
Although this may lead to multiple edges, the probability of this happening is negligible for large $N$. \\

\begin{figure}
\centering
\begin{subfigure}[b]{0.49\textwidth}
\centering
\includegraphics[width=\textwidth]{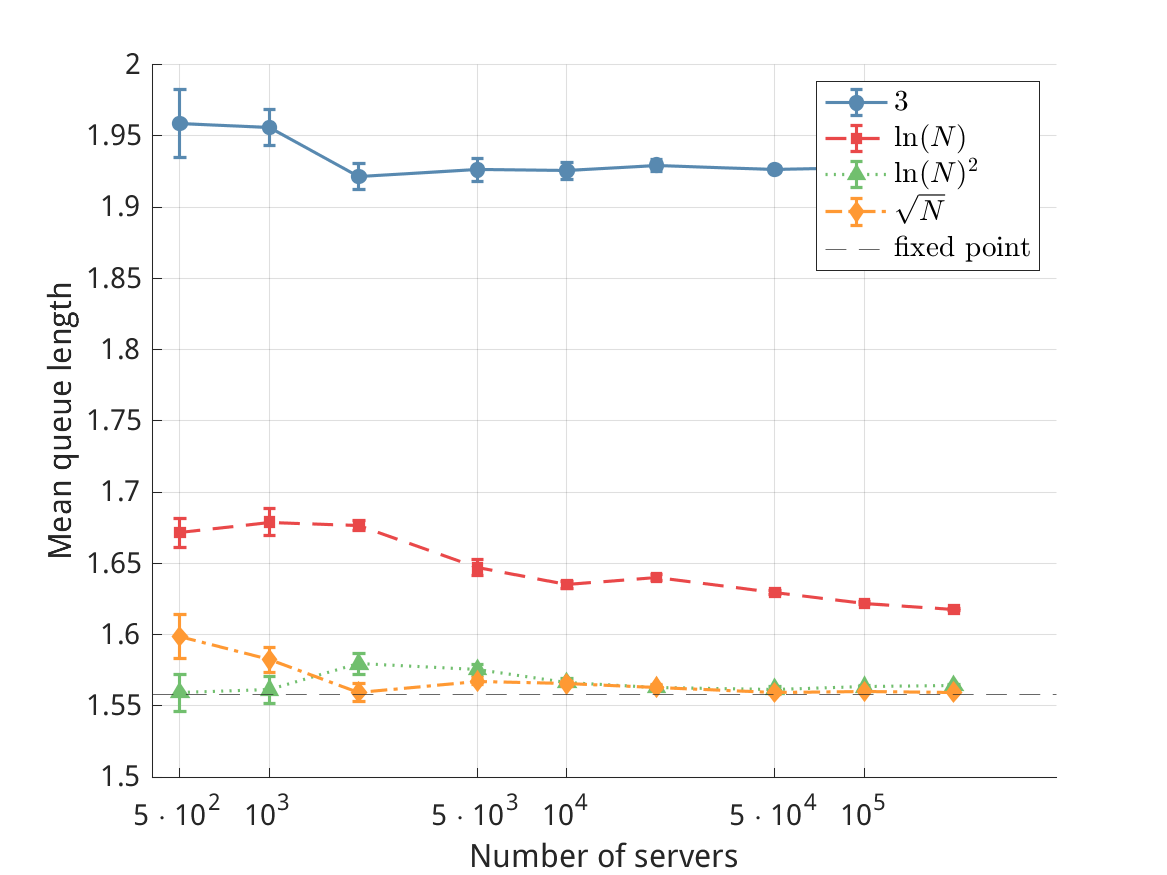}
\caption{Geometric graph}
\end{subfigure}
\hfill
\begin{subfigure}[b]{0.49\textwidth}
\centering
\includegraphics[width=\textwidth]{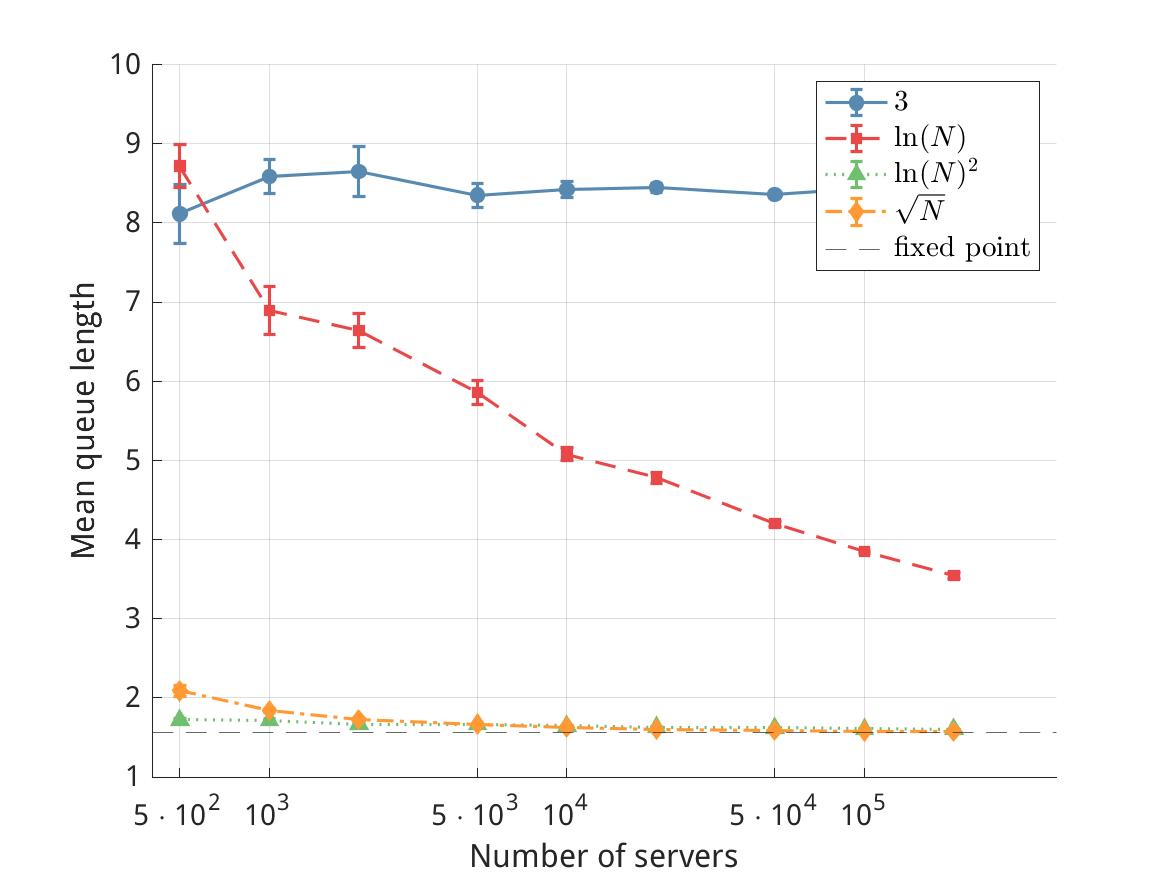}
\caption{Regular graph}
\end{subfigure}
\caption{The mean queue length in steady-state for a random regular bipartite graph and a random bipartite geometric graph for various degrees compared to the fixed point of the fluid limit.}
\label{fig:meanqueue}
\end{figure}

\noindent
\textbf{Mean queue length.} Figure \ref{fig:meanqueue} shows the mean queue length in steady-state for various (average) degrees. For the random bipartite geometric graphs, the mean queue length converges to the fixed point as $N \to \infty$ for an average degree of $(\ln(N))^2$ and $\sqrt{N}$ as expected by our main results. The mean queue length also seems to converge for an average degree of $\ln(N)$, albeit slowly. A rate of $\ln(N)$ is the edge case of our main result and, even though the mean queue length seems to converge, the tail of the occupancy is not double exponential (see Figure \ref{fig:occupancy}). For the random regular bipartite graphs, the mean queue length converges to the fixed point as $N \to \infty$ for a degree of $\ln(N)$, $(\ln(N))^2$ and $\sqrt{N}$. The mean queue length does not converge for a constant degree of $3$ in either case. Thus, the condition for the regular bipartite graph is both necessary and sufficient. \\

\begin{figure}
\centering
\begin{subfigure}[b]{0.49\textwidth}
\centering
\includegraphics[width=\textwidth]{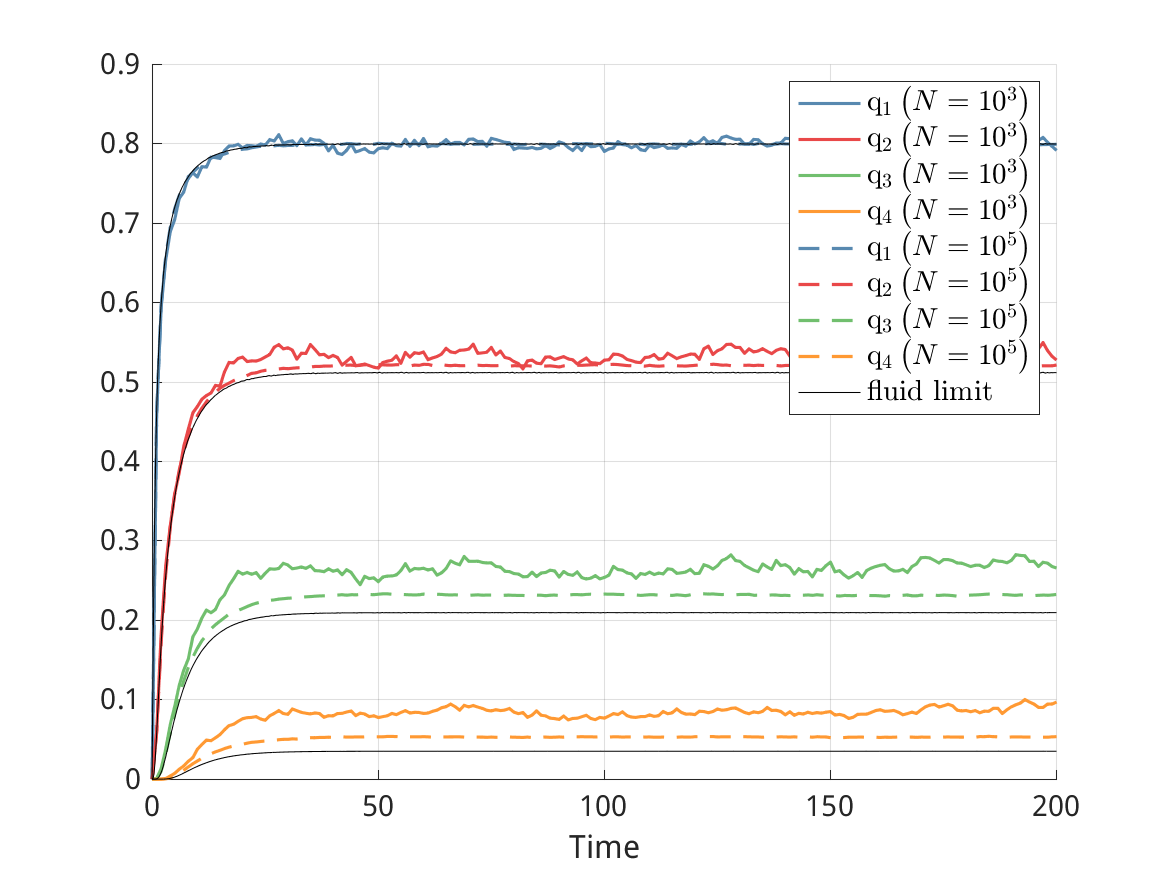}
\caption{Geometric ($\ln(N)^2$)}
\end{subfigure}
\hfill
\begin{subfigure}[b]{0.49\textwidth}
\centering
\includegraphics[width=\textwidth]{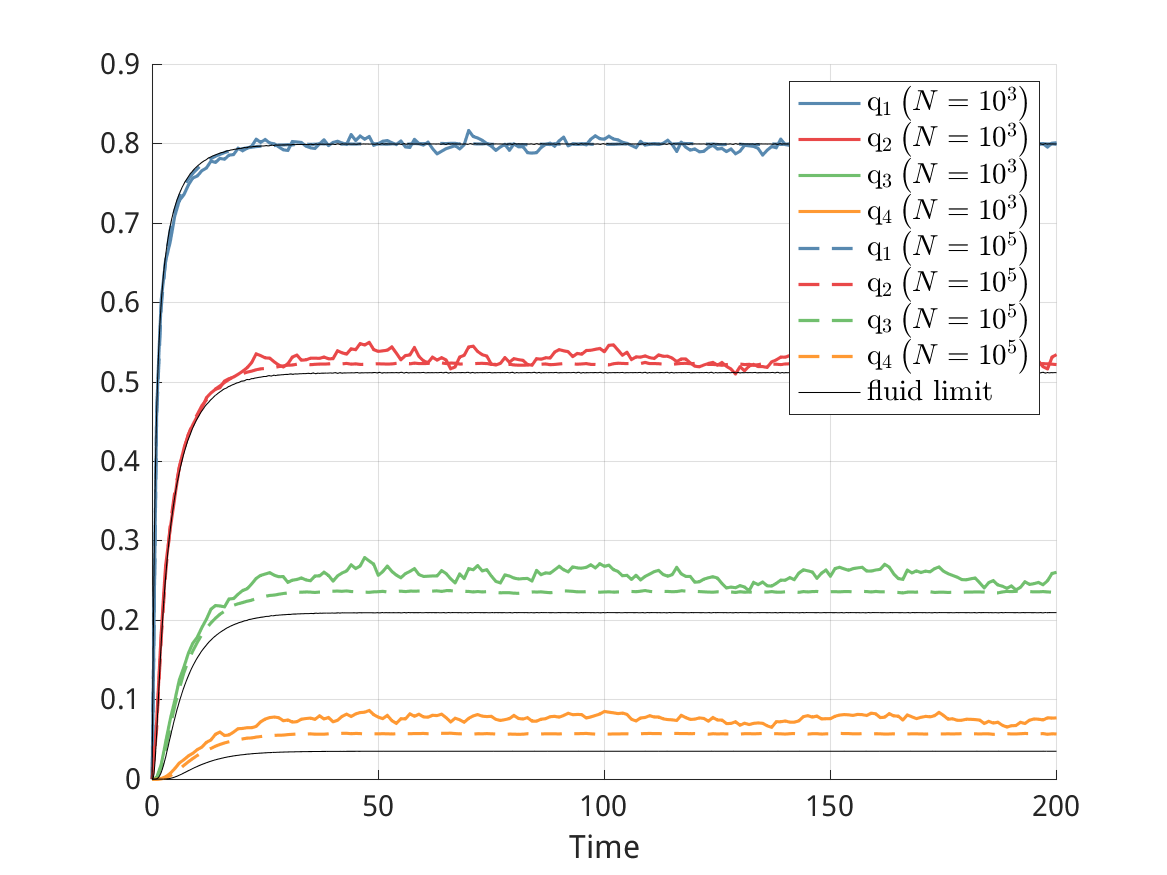}
\caption{Regular ($\ln(N)$)}
\end{subfigure}
\caption{The process-level limit of the occupancy process ($q_i(t)$) for a random bipartite geometric graph and a random regular bipartite graph and compared to the fluid limit, started from the empty state. The average degree is noted in parentheses.}
\label{fig:processlevel}
\end{figure}

\noindent
\textbf{Process-level limit from the empty state.} Figure \ref{fig:processlevel} shows the transient behavior of the system for two values of $N$, starting from the empty state. As $N$ increases the process remains close to the solution of ODEs, or the fluid limit, for both type of graphs. Note that the process still deviates slightly from the fluid limit, especially for $q_3(t)$ and $q_4(t)$, since the average degree grows only logarithmic in $N$, which directly impacts the convergence rate as in Theorem \ref{thm:process_level_limit}. \\

\begin{figure}
\centering
\begin{subfigure}[b]{0.49\textwidth}
\centering
\includegraphics[width=\textwidth]{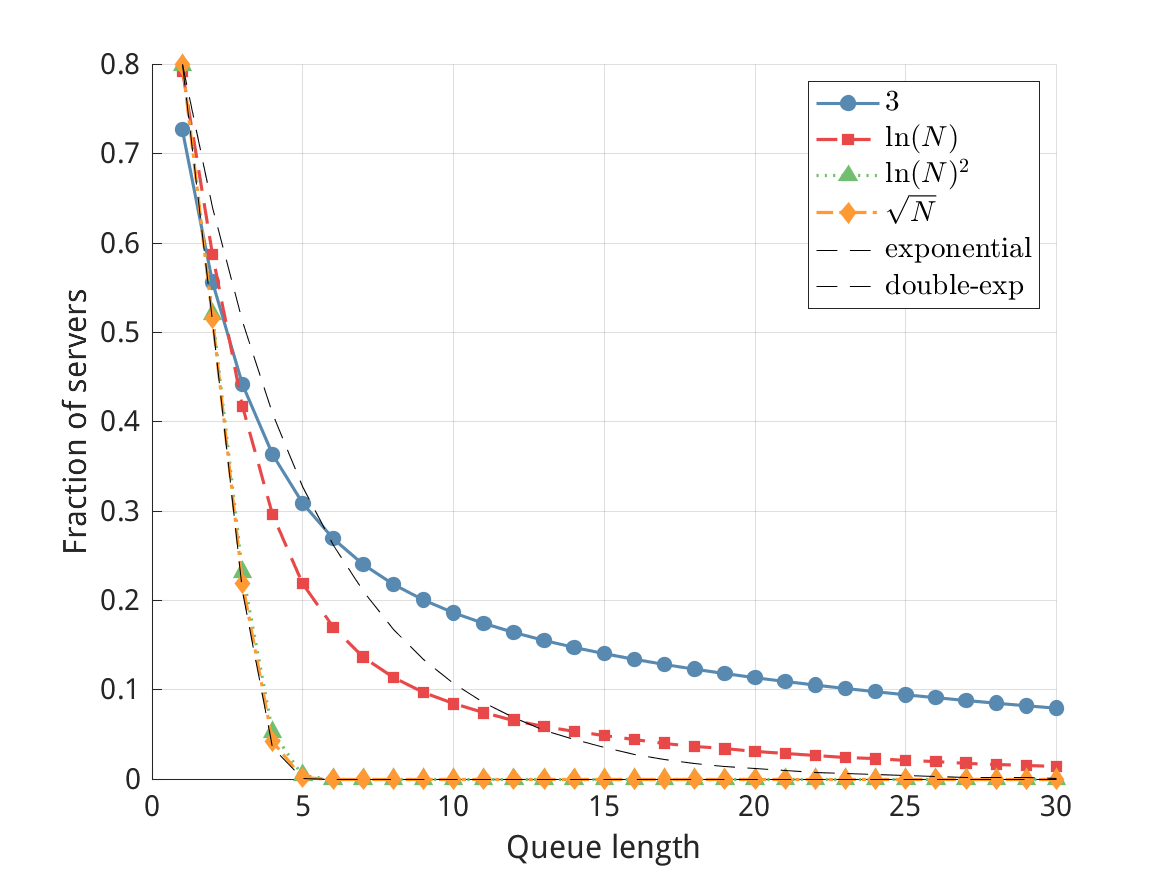}
\caption{Random geometric}
\end{subfigure}
\hfill
\begin{subfigure}[b]{0.49\textwidth}
\centering
\includegraphics[width=\textwidth]{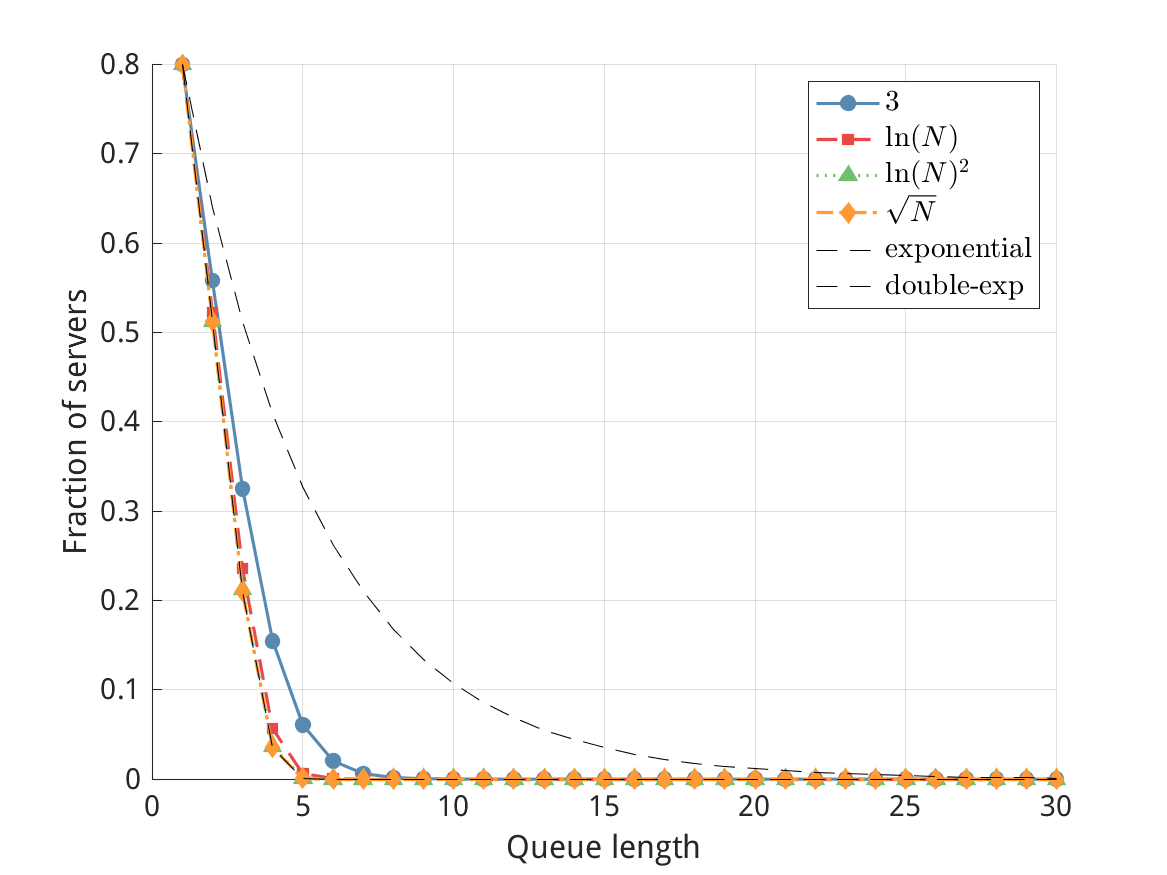}
\caption{Random regular}
\end{subfigure}
\caption{The occupancy in steady-state $(q_i(\infty))_{i \geq 1}$ for a random bipartite geometric graph and a random regular bipartite graph for various degrees compared to exponential and double-exponential tails for $N = 10^4$.}
\label{fig:occupancy}
\end{figure}

\noindent
\textbf{Exponential or double-exponential tail.} There are previous works that have asked whether a similar double-exponential tail of the queue lengths also holds for graphs of constant degree, such as a cycle~\cite{Gast15}. 
Figure~\ref{fig:occupancy} shows the occupancy $q_i(\infty)$ in steady-state for various degrees and for $N = 10^5$. The figure compares the occupancy to an exponential tail of $\lambda^i$ and a double exponential tail of $\lambda^{\frac{d^i-1}{d-1}}$. For the random bipartite geometric graph, the double exponential tail holds for an average degree of $(\ln(N))^2$ and $\sqrt{N}$ as expected by our main results. For an average degree of $3$ or $\ln(N)$, the system does not have the double exponential tail, and the performance even appears to be worse than the exponential tail (or random routing on a complete graph). 
For the random regular bipartite graph, the double exponential tail seems to holds for any choice of degree, even for a constant degree of $3$. However, in this case, the queue lengths have double exponential tail as $\lambda^{\frac{d^i-1}{d-1}}$ but with a slightly lower value of $d$. 
The question of whether it is possible to analytically characterize this value of $d$ remains a very interesting direction for future work, even for specific regular graphs with constant degree.

\section*{Acknowledgements}
The work was partially supported by the NSF grant CIF-2113027.

\def\UrlBreaks{\do\/\do-}
\bibliographystyle{apalike}
\bibliography{main,references-mukherjee}

\appendix

\section{Proof of Lemma~\ref{lem:asg_monotone}}
\label{sec:app-lem4}

\begin{proof}[Proof of Lemma~\ref{lem:asg_monotone}]
Fix any $0 \leq y < x \leq 1$. Then,
\begin{equation}
    \frac{x^d - y^d}{x - y}
    = \frac{\left( (x - y) + y \right)^d - y^d}{x - y}
    = \frac{\sum_{i = 0}^d \binom{d}{i} (x - y)^i y^{d-i} - y^d}{x - y}
    = \sum_{i = 1}^d \binom{d}{i} (x - y)^{i-1} y^{d-i}.
\end{equation}
Also,
\begin{equation}
\begin{multlined}
    \frac{\partial}{\partial y} \frac{x^d - y^d}{x - y}
    =  \frac{x^d - y^d}{(x - y)^2} - \frac{d y^{d-1}}{x - y}
    = \frac{\left( (x - y) + y \right)^d - d (x - y) y^{d-1} - y^d}{(x - y)^2} \\
    = \frac{\sum_{i = 0}^d \binom{d}{i} (x - y)^i y^{d-i} - d (x - y) y^{d-1} - y^d}{(x - y)^2}
    = \sum_{i = 2}^d \binom{d}{i} (x - y)^{i-2} y^{d-i}
    \geq 0.
\end{multlined}
\end{equation}
Then, by the mean value theorem, there exists $\xi \in [y_1, y_2]$ such that
\begin{equation}
\begin{multlined}
    \frac{x_1^d - y_1^d}{x_1 - y_1}
    = \sum_{i = 1}^d \binom{d}{i} y_1^{d-i} (x_1 - y_1)^{i-1}
    \leq \sum_{i = 1}^d \binom{d}{i} y_1^{d-i} (x_2 - y_1)^{i-1}
    = \frac{x_2^d - y_1^d}{x_2 - y_1} \\
    = \frac{x_2^d - y_2^d}{x_2 - y_2} - \frac{\partial}{\partial y} \frac{x_2^d - y^d}{x_2 - y} \bigg\rvert_{y = \xi} (y_2 - y_1)
    \leq \frac{x_2^d - y_1^d}{x_2 - y_1},
\end{multlined}
\end{equation}
which completes the proof of the lemma.
\end{proof}

\section{Proof of Lemma \ref{lem:q_martingale}}
\label{sec:q_martingale}

\begin{proof}[Proof of Lemma~\ref{lem:q_martingale}]
Fix any $t \geq 0$. To change the value of $Q_i^w(t)$, a task must arrive to a server $v \in \cN_w$ with queue length $i-1$ or a task must depart a server $v \in \cN_w$ with queue length $i$.

Fix any $v \in \cN_w$ with $X_v(t-) = i-1$ and let us compute the probability that a task is assigned to $v$. At the epoch time of an arrival, a task adopts a task type $w' \in W$ uniformly at random. The task is then routed to a server with queue length $i-1$ if and only if the system only samples servers with queue length at least $i-1$ and not only servers with queue length at least $i$, which happens with probability $q_{i-1}^{w'}(t-)^d - q_i^{w'}(t-)^d$. By symmetry, any server in $\cN_{w'}$ with queue length $i-1$ has the same probability of receiving the task and there are a total of $Q_{i-1}^{w'}(t-) - Q_i^{w'}(t-)$ of such eligible server. This results in a probability of
\begin{equation}
\begin{multlined}
    \frac{1}{M} \sum_{w' \in \cN_v} \frac{q_{i-1}^{w'}(t-)^d - q_i^{w'}(t-)^d}{Q_{i-1}^{w'}(t-) - Q_i^{w'}(t-)}.
\end{multlined}
\end{equation}

Now, fix any $v \in \cN_w$ with $X_v(t-) = i$ and let us compute the probability that a task departs $v$. At the epoch time of a potential departure, a server $v' \in V$ is chosen uniformly at random and a task departs if $v'$ has at least one task in its queue. This results in a probability of $1 / N$.

We describe the arrival and departure process as follows. Let $N(t)$ be a Poisson process of rate $(\lambda + 1) N$. An event of the process is either an arrival of type $w \in W$ with probability $\lambda / ((\lambda + 1) M)$ or a potential departure at server $v \in V$ with probability $1 / ((\lambda + 1) N)$, independent of the past. Note that this is equivalent to the model description introduced before. Hence, for any $h > 0$,
\begin{equation}
\begin{multlined}
\label{eq:q_diff_h}
    \E\left[ \Delta Q_i^w(t) \mid \cF_t \right]
    = \E\left[ \Delta Q_i^w(t) \mid \Delta N(t) = 1, \cF_t \right] \P\left( \Delta N(t) = 1 \right) \\
    \pm \E\left[ \Delta N(t) \mid \Delta N(t) \geq 2 \right] \P\left( \Delta N(t) \geq 2 \right) \\
    = \left( \frac{\lambda}{\lambda + 1} \sum_{\substack{v \in \cN_w \\ X_v(t) = i-1}} \frac{1}{M} \sum_{w' \in \cN_v} \frac{q_{i-1}^{w'}(t)^d - q_i^{w'}(t)^d}{Q_{i-1}^{w'}(t) - Q_i^{w'}(t)} - \frac{1}{\lambda + 1} \sum_{\substack{v \in \cN_w \\ X_v(t) = i}} \frac{1}{N} \right) (\lambda + 1) N h e^{-(\lambda + 1) N h} \\
    \pm \left( (\lambda + 1) N h + 2 \right) \left( (\lambda + 1) N h \right)^2,
\end{multlined}
\end{equation}
where $\Delta Q_i^w(t) := Q_i^w(t + h) - Q_i^w(t)$ and $\Delta N(t) := N(t + h) - N(t)$. Here, we use the shorthand notation $\pm x$ to denote a term in $[-x, x]$. Fix any $0 \leq s \leq t$. The equation above implies that
\begin{equation}
\begin{multlined}
\label{eq:q_deriv}
    \frac{d}{dt} \E\left[ Q_i^w(t) \mid \cF_s \right]
    = \lim_{h \downarrow 0} \frac{\E\left[ \E\left[ Q_i^w(t + h) - Q_i^w(t) \mid \cF_t \right] \mid \cF_s \right]}{h} \\
    = \E\Big[ \frac{\lambda N}{M} \sum_{\substack{v \in \cN_w \\ X_v(t) = i-1}} \sum_{w' \in \cN_v} \frac{q_{i-1}^{w'}(t)^d - q_i^{w'}(t)^d}{Q_{i-1}^{w'}(t) - Q_i^{w'}(t)} - \left( Q_i^w(t) - Q_{i+1}^w(t) \right) \mid \cF_s \Big],
\end{multlined}
\end{equation}
and hence, by the second fundamental theorem of calculus and Fubini's theorem,
\begin{equation}
\begin{multlined}
\label{eq:q_int}
    \E\left[ Q_i^w(t) - Q_i^w(s) \mid \cF_s \right]
    = \int_s^t \frac{d}{du} \E\left[ Q_i^w(u) \mid \cF_s \right] \diff u \\
    = \E\Big[ \int_s^t \frac{\lambda N}{M} \sum_{\substack{v \in \cN_w \\ X_v(u) = i-1}} \sum_{w' \in \cN_v} \frac{q_{i-1}^{w'}(u)^d - q_i^{w'}(u)^d}{Q_{i-1}^{w'}(u) - Q_i^{w'}(u)} - \left( Q_i^w(u) - Q_{i+1}^w(u) \right) \diff u \mid \cF_s \Big],
\end{multlined}
\end{equation}
which proves that $\E\left[ M_i^w(t) \mid \cF_s \right] = M_i^w(s)$. Also,
\begin{equation}
\begin{multlined}
    \lvert M_i^w(t) \rvert
    \leq \left\lvert Q_i^w(t) - Q_i^w(0) \right\rvert + \int_0^t \frac{\lambda N}{M} \sum_{\substack{v \in \cN_w \\ X_v(s) = i-1}} \sum_{w' \in \cN_v} \frac{q_{i-1}^{w'}(s)^d - q_i^{w'}(s)^d}{Q_{i-1}^{w'}(s) - Q_i^{w'}(s)} + \left( Q_i^w(s) - Q_{i+1}^w(s) \right) \diff s \\
    \leq d_w + \Big( \frac{\lambda N}{M} \sum_{v \in \cN_w} \sum_{w' \in \cN_v} \frac{d}{d_{w'}} + d_w \Big) t
    < \infty,
\end{multlined}
\end{equation}
by the mean-value theorem. This implies, in particular, that $M_i^w(t)$ is a square-integrable martingale.

We proceed by computing the quadratic variation of $M_i^w(t)$. As $Q_i^w(0)$ is a constant and the integral term is a continuous, finite variation process, it follows that $\left[ M_i^w \right]_t = \left[ Q_i^w \right]_t$. Furthermore, since $Q_i^w(t)$ is a finite variation process that is right-continuous with left limits, it follows that $\left[ Q_i^w \right]_t = \sum_{k = 1}^n (Q_i^w(t_k) - Q_i^w(t_k-))^2$, where $t_1, t_2, \dots, t_n$ are the (random) jump times of the process. Now, recall that the jumps of $Q_i^w(t)$ are always equal to one and hence $\left[ Q_i^w \right]_t$ must simply count the total number of jumps. Thus, a similar computation along the lines of \eqref{eq:q_diff_h} and \eqref{eq:q_deriv} yields
\begin{equation}
    \frac{d}{dt} \E\left[ \left[ Q_i^w \right]_t \right]
    = \E\left[ \frac{\lambda N}{M} \sum_{\substack{v \in \cN_w \\ X_v(t) = i-1}} \sum_{w' \in \cN_v} \frac{q_{i-1}^{w'}(t)^d - q_i^{w'}(t)^d}{Q_{i-1}^{w'}(t) - Q_i^{w'}(t)} + \left( Q_i^w(t) - Q_i^w(t) \right) \right].
\end{equation}
Then, applying the second fundamental theorem of calculus and Fubini's theorem as done in \eqref{eq:q_int} concludes the proof of the lemma.
\end{proof}

\section{Auxiliary lemmas}

\begin{lemma}
\label{lem:asg_lipschitz}
Fix any $0 \leq y_1 < x_1 \leq 1$ and $0 \leq y_2 < x_2 \leq 1$. Then, there exists $K > 0$ (depending only on $d$) such that
\begin{equation}
    \left\lvert \frac{x_1^d - y_1^d}{x_1 - y_1} - \frac{x_2^d - y_2^d}{x_2 - y_2} \right\rvert
    \leq K \left( \lvert x_1 - x_2 \rvert + \lvert y_1 - y_2 \rvert \right).
\end{equation}
\end{lemma}

\begin{proof}
Fix any $0 \leq y < x \leq 1$. Then,
\begin{equation}
    \frac{x^d - y^d}{x - y}
    = \frac{((x - y) + y)^d - y^d}{x - y}
    = \frac{\sum_{i = 0}^d \binom{d}{i} (x - y)^i y^{d-i} - y^d}{x - y}
    = \sum_{i = 1}^d \binom{d}{i} (x - y)^{i-1} y^{d-i}.
\end{equation}
Note that $\lvert x^d - y^d \rvert \leq d \lvert x - y \rvert$ by the mean value theorem. Therefore,
\begin{equation}
\begin{multlined}
    \left\lvert \frac{x_1^d - y_1^d}{x_1 - y_1} - \frac{x_2^d - y_2^d}{x_2 - y_2} \right\rvert
    \leq \sum_{i = 1}^d \binom{d}{i} \left\lvert (x_1 - y_1)^{i-1} y_1^{d-i} - (x_2 - y_2)^{i-1} y_2^{d-i} \right\rvert \\
    = \sum_{i = 1}^d \binom{d}{i} \left\lvert (x_1 - y_1)^{i-1} y_1^{d-i} - (x_1 - y_1)^{i-1} y_2^{d-i} + (x_1 - y_1)^{i-1} y_2^{d-i} - (x_2 - y_2)^{i-1} y_2^{d-i} \right\rvert \\
    \leq \sum_{i = 1}^d \binom{d}{i} \left( (d-i) \left\lvert y_1 - y_2 \right\rvert + (i-1) \left\lvert (x_1 - y_1) - (x_2 - y_2) \right\rvert \right) \\
    \leq \sum_{i = 1}^d \binom{d}{i} \left( (i-1) \left\lvert x_1 - x_2 \right\rvert + (d-1) \left\lvert y_1 - y_2 \right\rvert \right) \\
    \leq (2^d - 1) (d-1) \left( \left\lvert x_1 - x_2 \right\rvert + \left\lvert y_1 - y_2 \right\rvert \right),
\end{multlined}
\end{equation}
which completes the proof of the lemma.
\end{proof}

\begin{lemma}
\label{lem:martingale_vanishes}
Let $M_i^w(t)$ be as defined in Lemma \ref{lem:q_martingale}. Then, for all $t \geq 0$,
\begin{equation}
    \E\left[ \sum_{i = 1}^\infty \frac{1}{M} \sum_{w \in W} \sup_{s \in [0, t]} \frac{M_i^w(s)^2}{d_w^2} \right]
    \leq 4 t (\rho_0 d + 1) \gamma(G).
\end{equation}
\end{lemma}

\begin{proof}
Fix any $i \in \N$, $w \in W$ and $t \geq 0$. Note that $M_i^w(t)$ is a square-integrable martingale and therefore, by Doob's martingale inequality,
\begin{equation}
\begin{multlined}
    \E\Big[ \sup_{s \in [0, t]} M_i^w(s)^2 \Big]
    \leq 4 \E\left[  M_i^w(t)^2 \right]
    = 4 \E\left[ \left[ M_i^w \right]_t \right] \\
    = 4 \E\Big[ \int_0^t \frac{\lambda M}{N} \sum_{\substack{v \in \cN_w \\ X_v(s) = i-1}} \sum_{w' \in \cN_v} \frac{q_{i-1}^{w'}(s)^d - q_i^{w'}(s)^d}{Q_{i-1}^{w'}(s) - Q_i^{w'}(s)} + \left( Q_i^w(s) - Q_{i+1}^w(s) \right) \diff s \Big] \\
    \leq 4 \E\Big[ \int_0^t \sum_{\substack{v \in \cN_w \\ X_v(s) = i-1}} \frac{\lambda M}{N} \sum_{w' \in \cN_v} \frac{d}{d_{w'}} + \left( Q_i^w(s) - Q_{i+1}^w(s) \right) \diff s \Big] \\
    \leq 4 \E\left[ \int_0^t \rho_0 d \left( Q_{i-1}^w(s) - Q_i^w(s) \right) + \left( Q_i^w(s) - Q_{i+1}^w(s) \right) \diff s \right],
\end{multlined}
\end{equation}
where we use Lemma \ref{lem:q_martingale} in the second equality and the mean-value theorem in the second inequality. The equation above implies, by the monotone convergence theorem and Fubini's theorem,
\begin{equation}
\begin{multlined}
    \E\left[ \sum_{i = 1}^\infty \frac{1}{M} \sum_{w \in W} \sup_{s \in [0, t]} \frac{M_i^w(s)^2}{d_w^2} \right] \\
    \leq 4 \int_0^t \frac{1}{M} \sum_{w \in W} \frac{1}{d_w^2} \sum_{i = 1}^\infty \E\left[ \rho_0 d \left( Q_{i-1}^w(s) - Q_i^w(s) \right) + \left( Q_i^w(s) - Q_{i+1}^w(s) \right) \right] \diff s \\
    \leq 4 \int_0^t \frac{1}{M} \sum_{w \in W} \frac{\rho_0 d + 1}{d_w} \diff s
    = 4 t \left( \rho_0 d + 1 \right) \gamma(G),
\end{multlined}
\end{equation}
which completes the proof of the lemma.
\end{proof}

\end{document}